\documentclass[preprint,11pt]{imsart}
\RequirePackage[T1]{fontenc}
\usepackage{color}
\usepackage[latin1]{inputenc}
\usepackage[T1]{fontenc}

\usepackage{amsmath,amssymb,graphicx}
\usepackage{epstopdf}
\usepackage{color}
\RequirePackage{amsfonts,amsthm,amsmath}
\RequirePackage[colorlinks,citecolor=blue,urlcolor=blue]{hyperref}
\usepackage{a4wide}
\usepackage{here}
\usepackage{float}
\usepackage{amsmath}
\usepackage{bbm}

\newcommand{\nbR}{\mathbb{R}}
\newcommand{\nbN}{\mathbb{N}}

\newcommand{\nbC}{\mathbb{C}}

\newcommand{\nbu}{\mathbbm{1}}

\newcommand{\nbP}{\mathbb{P}}

\newcommand{\nbE}{\mathbb{E}}

\renewcommand{\L}{{\cal L}}

\newcommand{\ph}{\varphi}

\renewcommand{\l}{\ell}

\newtheorem{theorem}{{\bf Theorem}}
\newtheorem{lemm}[theorem]{{\bf Lemma}}
\newtheorem{cor}[theorem]{{\bf Corollary}}
\newtheorem{remark}[theorem]{Remark}
\newtheorem{example}[theorem]{Example}
\newtheorem{prop}[theorem]{\bf Proposition}

\def\EE{{\mathbb E}}

\newcommand{\Cov}{\mathbb{C}ov}

\begin{document}

\begin{frontmatter}

\title{On a multivariate renewal-reward process involving time delays and discounting: Applications to IBNR process and infinite server queues
}
\runtitle{Multivariate renewal-reward process}

\begin{aug}
\author{\fnms{\Large{Landy}}
\snm{\Large{Rabehasaina}}
\ead[label=e2]{lrabehas@univ-fcomte.fr}} \and \
\author{\fnms{\Large{Jae-Kyung}}
\snm{\Large{Woo}}
\ead[label=e1]{j.k.woo@unsw.edu.au}}
\runauthor{L.Rabehasaina and J.K.Woo}

\address{\hspace*{0cm}\\
Laboratory of Mathematics, University Bourgogne Franche Comt\'e,\\
16 route de Gray, 25030 Besan\c con cedex, France.\\[0.2cm]
\printead{e2}}

\address{\hspace*{0cm}\\
School of Risk and Actuarial Studies, Australian School of Business,\\ University of New South Wales, Australia.\\[0.2cm]
\printead{e1}}
\end{aug}

\vspace{0.5cm}

\begin{abstract}

This paper considers a particular renewal-reward process with multivariate discounted rewards (inputs) where the arrival epochs are adjusted by adding some random delays. Then this accumulated reward can be regarded as multivariate discounted \emph{Incurred But Not Reported} (IBNR) claims in actuarial science and some important quantities studied in queueing theory such as the number of customers in $G/G/\infty$ queues with correlated batch arrivals. We study the long-term behavior of this process as well as its moments. Asymptotic expressions and bounds for the quantities of our interest, and also convergence result for the distribution of this process after renormalization, are studied, when interarrival times and time delays are light tailed. Next, assuming exponentially distributed delays, we derive some explicit and numerically feasible expressions for the limiting joint moments. In such case, for an infinite server queue with renewal arrival process, we obtain limiting results on the expectation of the workload, and the covariance of queue size and workload. Finally, some queueing theoretic applications are provided.

\end{abstract}

\end{frontmatter}

\normalsize

 \section{Introduction and notation}\label{intro}

Many situations in which processes restart probabilistically at renewal instants and there are non-negative rewards associated with each renewal epoch, are well described by a multivariate renewal-reward process. For example, a multivariate reward function can be viewed as an accumulated cost from different types of properties or infrastructures caused by a single catastrophe event, which is of interest in actuarial science and reliability analysis. The asymptotic distribution and the asymptotic expansion for the covariance function of the rewards were studied by \cite{PNT15} and \cite{AB17} who extended the result of \cite{BS75} to multivariate case.
In the context of actuarial science, much research about the aggregate discounted
claims has been done on its moment under renewal claim arrival processes. For example, \cite{LA11}, \cite{LG01a}, \cite{LG01b}, and \cite{LGW10} analyzed the renewal process, and \cite{WC13} looked at the dependent renewal process.

In this paper, we assume that there are time lags added to the original arrival times of renewal process. These delayed renewal epochs allow us to study the quantities related to infinite server queues with correlated batch arrivals and multivariate \emph{Incurred But Not Reported} (IBNR) claims where there is a delay in reporting or payment for claims. Furthermore, rewards are accumulated at a discounted value. A direct application to some problems in infinite server queues includes the case, for example, when the bulk size random variable is multivariate (i.e. correlated) and the service time distribution is dependent on the type of input. In this case a multivariate reward function incorporating time delays up to time $t$ (with zero discounting factor) is essentially the number of customers in the system up to time
$t$. In the infinite server queues with multiple batch Markovian arrival streams, a time-dependent matrix joint generating function of the number of customers in the system was derived by \cite{MT02}.
For the univariate case, IBNR claim count with batch arrivals was considered by \cite{GLW13} and the total discounted IBNR claim amount was studied by \cite{LWX16}. For the multivariate case, \cite{W15} provided expressions for the joint moments of multivariate IBNR claims which are recursively computable. For the number of IBNR claims, a direct relation to the number of customers in the infinite server queues with batch arrivals is well known as discussed in the literature, e.g. \cite{Karlsson}, \cite{LWX16}, \cite{WD01}, \cite{WDC02}, \cite{WD09}. The transient behavior of a distribution of the number of customer in various multichannel bulk queues was studied in \cite{CT83}. See also \cite{BR69} for example.

Let us introduce the model more precisely. We shall suppose that the batch arrival process
$\{N_t\}_{t\geq 0}$ is a renewal process with a sequence of independent and identically distributed (iid) positive continuous random variables (rv)s $(T_i)_{i\in \nbN}$ representing the arrival time of the $i$th batch with $T_0 \equiv 0$. Let $\tau_i=T_i-T_{i-1}$ be the interarrival time of the $i$th batch with a common probability density function (pdf) $f$, distribution $F$, and the Laplace transform $\L^\tau(u)=\nbE[e^{-u\tau_1}]$ for $u\ge 0$. Also we denote the renewal function and renewal density $t\mapsto m(t):=\nbE[N_t]$ and $u(t)=\frac{d}{dt}m(t)$ respectively. Each batch arrival contains several ($k$) types of inputs which may simultaneously occur from the same renewal event (e.g. \cite{PNT15}, \cite{W15}). Let us denote the $j$-type of input from the $i$th batch as $X_{i,j}$ where $\{(X_{i,1},\ldots,X_{i,k})\}_{i\in\nbN}$ is a sequence of iid random vectors. A vector for generic multivariate input variables is denoted as $X=(X_1,X_2,\ldots,X_k)$.
Here multivariate input values are assumed to be dependent on the occurrence time and/or the adjusted time by adding a random delay. This time delay for the $j$-type of input from the $i$th batch is denoted by $L_{i,j}$ where $(L_{i,j})_{i\in\nbN}$ is a sequence of iid random variables with a common cumulative distribution function $W_j(t)=1-\overline{W}_j(t)$, and such that $(L_{i,j})_{i\in\nbN, j=1,...,k}$ is a sequence of independent random variables. A generic time delay rv for the $j$-type of input is denoted by $L_j$. For the sake of simplicity let us assume a constant force of interest $\delta$ to discount input values to time 0,
and define the following discounted compound delayed process
\begin{equation}\label{Zdt}
Z(t)=Z(t,\delta)=(Z_1(t),\ldots,Z_k(t)),\qquad t\ge 0,
\end{equation}
where
\begin{equation}\label{Zjt}
Z_j(t):=\sum_{i=1}^{N_t} e^{-\delta(T_i+L_{i,j})} X_{i,j}\nbu_{\{ T_i+L_{i,j}>t\}}= \sum_{i=1}^{\infty} e^{-\delta(T_i+L_{i,j})} X_{i,j}\nbu_{\{T_i\leq t< T_i+L_{i,j}\}},\quad j\in\{1,\ldots,k\}.
\end{equation}
Here, we can interpret the process $\{Z(t)\}_{ t\ge 0}$ in two different ways. The first one is related to actuarial science: we suppose that aggregate claim amounts (or claim severities) $X_{i,j}$ in the branch $j\in \{1,\ldots,k \}$ of an insurance company is caused by the event arriving at time $T_i$. Instead of being dealt with immediately, they are (within a batch) subject to a delay $L_{i,j}$ until being reported. $Z_j(t)$ then represents the discounted total claim amounts of such IBNR claims in the branch $j$. The second one is related to queueing theory: let us consider a single queue containing $k$ types of customers in an infinite-server queue model. Here customers arrive according to a renewal process $\{N_t\}_{t\geq 0}$ with corresponding arrival times $ (T_i)_{i\in\nbN}$. At each arrival instant $T_i$ a batch of correlated customers $(X_{i,1},\ldots,X_{i,k})$ enters the system, with $X_{i,j}\in\nbN$. For each customer of class $j\in\{1,\ldots,k\}$ (of which number is $X_{i,j}$) the service time $L_{i,j}$ is the same. The service times $(L_{i,j})_{i\in \nbN,\ j=1,\ldots,k}$ are thus assumed to be independent, although $L_{i,1},\ldots,L_{i,k}$ possibly have different distributions, i.e. service times are different  according to the type of customer class. For example, if $\delta=0$, the model is reduced to that of $G/G/\infty$ queues with multiple types of customer classes in a batch. As an illustration, let us look at the particular case where $(X_1,\ldots,X_k)$ follows a multinomial distribution with parameters $M\in\nbN^*$ and a probability vector $(p_1,\ldots,p_k)$ where $p_j\ge 0$ and $\sum_{j=1}^k p_j=1$. This models a situation where at every instant $T_i$ exactly $M$ customers arrive, each of which belongs to class $j$ with probability $p_j$. Then $X_j$ represents the number of customers of class $j$ in this batch. See Figure \ref{example_queue}.
\begin{figure}[!hbtp]%
\centering
\includegraphics[scale=0.7]{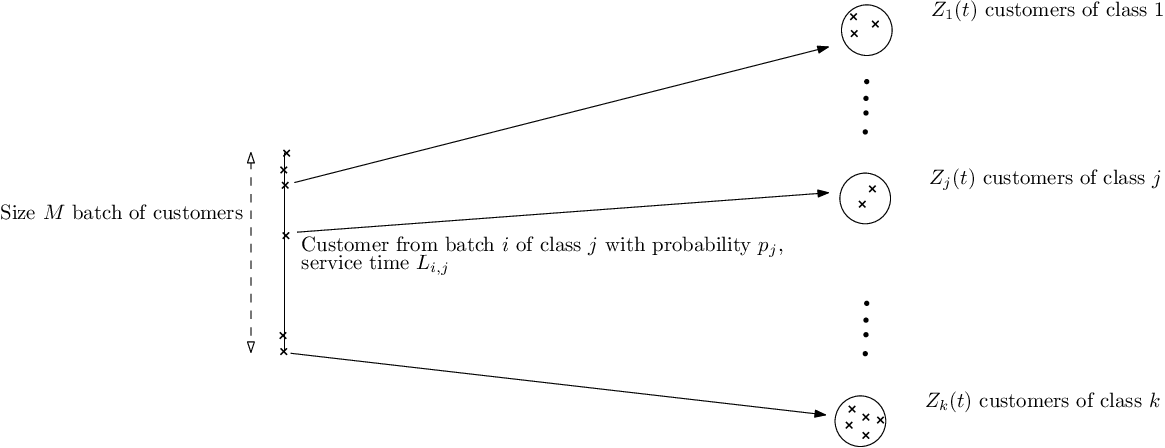}%
\caption{\label{example_queue} The $G/G/\infty$ queue with multinomial distributed classes batches $(X_1,\ldots,X_k)$.}
\end{figure}
The simplest scenario is when $M=1$, where each customer arrives according to renewal process $\{N_t\}_{t\ge 0}$, and belongs to class $j$ with probability $p_j$. Because of these two alternative interpretations in actuarial science and queueing theory as explained above, we will refer the $L_{i,j}$'s as either "delay" or "service" times, without ambiguity.

We note that it is usually difficult to derive a distribution for this discounted compound delayed process $Z(t)$ since there is no concrete representation for an inversion of the complicated moment generating function (mgf) for this quantity in a general arrival process  $\{N_t\}_{t\geq 0}$. In this sense, it is appealing to study the long-term behavior of the process in terms of its moment and distribution. From \cite{W15}, explicit expressions for the joint moments of $Z(t)=(Z_1(t),...,Z_k(t))$ are recursively obtainable. However, an analytic expression of the lower moment which appears in its integral term, is required for the calculation of the higher moment. Also, it is necessary to know an explicit form of the renewal density $u(t)$ for the evaluation of this moment. Therefore, our objective here is to develop simpler approximation methods such as asymptotics and bound results for the joint moments of $Z(t)$. To the best of our knowledge, these kinds of approximation approaches have never been developed in the analysis of a multivariate renewal-reward process with discounted inputs and time delays. Also, a relationship between multivariate discounted IBNR claim process and quantities studied in infinite server queues with correlated batch arrivals and a discounting factor is firstly exploited in this paper. Moreover, we shall also consider the case with exponential time delays in a general arrival process and provide asymptotic results for the joint moments. In this case, for light tailed interclaim time and single input, we are able to quantify the approximation precision by providing many terms for the asymptotics for the first order moment of our process. We note that this approach was previously found in \cite[Lemma 1]{BS75} where a 2-term asymptotic expression for a general renewal reward process without delays was provided, see also \cite{PNT15} and \cite{AB17} for an expansion of the covariance.
   In particular,  some asymptotic results regarding queueing theoretic applications such as the workload in the $G/M /\infty$ system, are obtained. 

In most cases in this paper, we suppose that the discounted factor $\delta$ is real and non negative because of its discounting role. However it has to be pointed out that, mathematically speaking, Definitions (\ref{Zdt}) and (\ref{Zjt}) can in some cases be extended to some {\it complex} $\delta$, as will be the case in Section \ref{sec_single_queue} where $\delta\in\nbC$ is needed for technical purposes. It will also be convenient to define the process $\tilde{Z}(t)=\tilde{Z}(t,\delta)=(\tilde{Z}_1(t),\ldots,\tilde{Z}_k(t))=e^{\delta t}Z(t)$, i.e.
\begin{equation}\label{def_Z_tilde}
\tilde{Z}_j(t)=\sum_{i=1}^{N_t} e^{\delta(t-T_i-L_{i,j})} X_{i,j}\nbu_{\{ T_i+L_{i,j}>t\}},\quad j\in\{1,\ldots,k\}.
\end{equation}
Although $\tilde{Z}(t)$ does not have a direct actuarial or queueing interpretation, it will turn out that most results will concern this process rather than $Z(t)$.

\noindent {\bf Notation. } For $n=(n_1,\ldots,n_k)\in \nbN^k$, the $n$th joint moments for $Z(t)$ and its mgf are respectively denoted as
\begin{eqnarray}
M_n(t)&=&\nbE \bigg[\prod^k_{j=1}Z_j^{n_j}(t)\bigg],\qquad t\geq 0,~~n=(n_1,\ldots,n_k)\in \nbN^k ,\label{Mnt}\\
\psi(s,t)&=& \nbE\Big[e^{<s,Z(t)>}\Big],\qquad s=(s_1,\ldots,s_k)\in \nbR^k,
\label{mgf}
\end{eqnarray}
where $<\!\cdot ,\cdot \! >$ is the euclidian scalar product. For notational convenience, we let, for all $n=(n_1,\ldots,n_k)\in \nbN^k$ and $t\ge 0$,
\begin{eqnarray}
\eta_n&:=&\sum_{i=1}^k n_i,\nonumber\\
\tilde{M}_n(t)&:=& e^{\eta_n \delta t}M_n(t)=\nbE \bigg[\prod^k_{j=1}\tilde{Z}_j^{n_j}(t)\bigg],\label{M_tilde}\\
\tilde{\psi}(s,t)&=&\nbE\Big[e^{<s,\tilde{Z}(t)>}\Big],\qquad s=(s_1,\ldots,s_k)\in \nbR^k.\label{mgf_Z_tilde}
\end{eqnarray}
We let ${\bf 0}=(0,\ldots,0)$ the zero vector in $\nbN^k$, and we define the natural partial order on set $\nbN^k$ as follows. We say that two vectors $\ell$ and $n$ in $\nbN^k$ verify $\ell<n$ if $\ell_i\le n_i$ for all $i=1,\ldots,k$ and $\ell_i<n_i$ for (at least) an $i$, i.e. $\eta_n> \eta_\l$. Let us introduce, for all $n\in \nbN^k$,
$$C_{\ell,n}:=\{ j=1,\ldots,k|\ \ell_j<n_j\}\subset \{1,\ldots,k \}.$$
We will denote by $n(i)\in\nbN^k$ the vector of which $j$th entry is $\delta_{i,j}$ where $\delta_{i,j}$ is the Kronecker delta function.

It is convenient to introduce the function $t\mapsto \ph_{\l,n}(t)$ for $\l < n$,
\begin{equation}
\ph_{\l ,n}(t)= \nbE\bigg[ e^{(\eta_n-\eta_\ell) \delta (t-\tau_1)}\tilde{M}_\ell(t-\tau_1)\prod_{j\in C_{\ell ,n}} \overline{\omega}_{(n_j-\ell_j)\delta,j}(t-\tau_1)
.\nbu_{[\tau_1<t]}\bigg],
\label{phi_l}
\end{equation}
where
\begin{equation}\label{bomega}
\overline{\omega}_{\delta,i}(t)=\int^\infty_t e^{-\delta y}dW_i(y).
 \end{equation}
Following \cite{W15}, we define $\tilde{b}_n(t)$ by
\begin{equation}
\tilde{b}_n(t)=\sum_{\ell < n} {{n_1}\choose{\ell_1}}\cdots{{n_k}\choose{\ell_k}}
\nbE\bigg[ \prod_{j=1}^k X_j^{n_j-\ell_j}\bigg]\ph_{\ell,n}(t).
\label{exp_tilde_b}
\end{equation}
Throughout the paper, ${\cal E}(\mu)$ denotes an exponential distribution with a mean $1/\mu$. We denote $|A|$ as the cardinal of $A$ for any finite set $A$. 

We assume that a vector $X$ admits joint moments of all order. We recall that a rv $Y\ge 0$ has New Better than Used (NBU) distribution if its survival function satisfies $\nbP(Y>x+y)\le \nbP(Y>x)\nbP(Y>y)$ for all non negative $x$ and $y$. Lastly, we denote assumptions {\bf (A1)}, {\bf (A1')} and {\bf (A2)} by:
\begin{eqnarray*}
{\bf (A1)}&\quad &\mbox{The pdf } f(\cdot) \mbox{ of interarrival time }\tau_1 \mbox{ is bounded,}\\
{\bf (A1')}&\quad & \mbox{interarrival time } \tau_1\mbox{ is light tailed: }\exists R>0,\ \int_0^\infty e^{Rx}dF(x)=\EE[e^{R\tau_1}]<+\infty ,\\
{\bf (A2)}&\quad & \exists M >0 \mbox{ such that } \forall j=1,\ldots,k,\ 0\le X_j\le M \mbox{ a.s.},\\
& & \mbox{or $X_j$ belongs to the NBU class}.
\end{eqnarray*}
It is noted that {\bf (A2)} is substantive in several queueing and actuarial applications. One way of viewing the upper bounded condition in queueing theory is to consider the number of arriving customers being fixed or limited (as illustrated in the example in Figure \ref{example_queue}). When the claim severity distribution follows a general family of NBU classes, some interesting applications in relation to reinsurance premium calculation are discussed in \cite[Section 3.1]{Kayid07}.

An important consequence of {\bf (A1)} is the following result, of which proof is given at the beginning of Section~\ref{sec:Proofs}.
\begin{lemm}\label{lemma_density_upper_bound}
If {\bf (A1)} holds then the associated renewal function $m:t\ge 0\mapsto m(t)=\nbE[N_t]$ admits a density $u(t)$, which verifies
\begin{equation}\label{Expression_density_renewal}
u(t)=\frac{d}{dt}m(t)=\sum_{j=0}^\infty f^{\star (j)}(t).
\end{equation}
Besides, this density is upper bounded: There exists $C>0$ such that
\begin{equation}\label{redenupper}
u(t)\le C,\qquad \forall t\ge 0.
\end{equation}
\end{lemm}
\begin{remark}\label{rem_constant_C}
\normalfont The existence of upper bound $C$ in (\ref{redenupper}) in the previous lemma is proved only from a theoretical point of view. We remark that this constant can be easily found in some cases such as Poisson and Erlang processes as will be seen in Example \ref{Erlang2Two}. Otherwise, some bound results for the renewal density $u(t)$ can be utilized to find $C$ when the interclaim time distribution has some particular properties e.g. has an Increasing Failure Rate (IFR) and/or has finite support (both of these conditions implying the required condition ({\bf A1})). For example, \cite[Proposition 4.1]{LP17} and \cite[Section 8.3, Corollary 8.7]{WW17} yields such an explicit bound when $F$ has Decreasing Failure Rate (DFR) and is New Worse than Used (NWU) with upper bounded failure rate respectively.
\end{remark}

\noindent{\bf Structure of paper.} For ease of presentation, all main results are given in Sections \ref{sec_renewal_equation}, \ref{sec_conv_moments_distribution}, \ref{sec_expo} and \ref{sec_single_queue}, and all the proofs are placed in Section \ref{sec:Proofs}. Section \ref{sec_renewal_equation} recalls the results from \cite{W15} that are used throughout the paper, with some immediate applications when interarrivals are exponentially distributed. Section \ref{sec_conv_moments_distribution} addresses the general case where interarrival and delays have arbitrary distributions, in which case one proves convergence of moments of $\tilde{Z}(t)$ (Proposition \ref{prop_asymptotics}) as well as convergence in distribution when {\bf (A1)} and {\bf (A2)} holds (Theorem \ref{theo_conv_distrib}). Section \ref{sec_expo} concerns the case where delays are exponentially distributed (Theorem \ref{expser}). Particular focus is made in Section \ref{sec_single_queue} when $k=1$ with exponentially distributed delays: we first give an asymptotic expansion for $\tilde{M}_1(t)$ as $t\to\infty$ when {\bf (A1')} holds (Theorem \ref{theorem_expansion}). In the subsequent subsection, this result is utilized to obtain asymptotic moments for the workload of the $G/M/\infty$ queue when {\bf (A1)} and {\bf (A1')} hold (Theorem \ref{workload}). In both those latter sections, we compare the results to the existing queueing literature, particularly those from Tak\'acs \cite{T62}. Finally, in Section \ref{sec:App}, an attempt is made to put some emphasis on the fact that the generality of the model yields interesting applications.

\section{Renewal equations: General and Exponential interarrival times}\label{sec_renewal_equation}
The aim of this section is to briefly review the results obtained in \cite{W15} that will be the starting point of most of the results in the present paper, and to recover some particular results when claims arrive according to a Poisson process. Following notation in \cite[Section 3.3]{W15}, we let for all $t\ge 0$ and $s=(s_1,\ldots,s_k)\in \nbR^k$,
\begin{eqnarray*}
M^*_{t,X}(s)&:=&\nbE \left[ \exp \left(\sum_{j=1}^k s_je^{-\delta L_{i,j}}X_{i,j}\nbu_{[L_{i,j}>t]}\right) \right]\\
&=&\int_0^\infty \cdots \int_0^\infty \nbE \left[ \exp \left(\sum_{j=1}^k s_je^{-\delta v_j}X_{i,j}\nbu_{[v_j>t]}\right) \right]dW_1(v_1)\cdots dW_k(v_k).
\end{eqnarray*}
From \cite[Section 3.3]{W15}, we know that the mgf of $Z(t)$ in (\ref{mgf}) satisfies
\begin{equation*}\label{expression_psi}
\psi(s,t)=\nbE \bigg[ \prod_{i=1}^{N_t}M^*_{t-T_i,X}(e^{-\delta T_i}s)\bigg],
\end{equation*}
and from (36) of \cite{W15}, (\ref{Mnt}) is recursively obtained as 
\begin{equation}\label{Eq36}
M_n(t)=
\sum_{ \ell < n} {{n_1}\choose{\ell_1}}\cdots{{n_k}\choose{\ell_k}} \nbE\bigg[ \prod_{j=1}^k X_j^{n_j-\ell_j}\bigg] \int^t_0 e^{-\eta_{\ell}\delta y} M_{\ell}(t-y)
\bigg[\prod_{j\in C_{\ell ,n}} \overline{\omega}_{(n_j-\ell_j)\delta,j}(t-y)\bigg]dm(y),
\end{equation}
and in particular, when $n=n(i)$, it reduces to
\begin{equation}\label{Eq36ni}
M_{n(i)}(t) = \nbE[X_i].\, \int_0^t e^{-\delta y}\overline{\omega}_{\delta, i}(t-y)dm(y)= \nbE[X_i].\, e^{-\delta t}\int_0^t e^{\delta (t-y)}\overline{\omega}_{\delta, i}(t-y)dm(y).
\end{equation}
Also, from \cite[Theorem 3]{W15}, $\tilde{M}_n(t)$ defined in (\ref{M_tilde}) satisfies
\begin{eqnarray}
\tilde{M}_{n(i)}(t) &=& \nbE[X_i].\, \int_0^t e^{\delta (t-y)}\overline{\omega}_{\delta, i}(t-y)dm(y),\quad i=1,\ldots,k,\nonumber\\
\tilde{M}_n(t)&=& \tilde{b}_n(t) +\tilde{M}_n\star F(t),\qquad t\geq 0,\quad n\in\nbN^k\!\setminus \!\{n(i),\ i=1,\ldots,k\},
\label{renewal_M_tilde}
\end{eqnarray}
where $\overline{\omega}_{\delta, i}(t)$ and $\tilde{b}_n(t)$ are respectively given by (\ref{bomega}) and \eqref{exp_tilde_b}.
It is standard that the solution to \eqref{renewal_M_tilde} is given by $\tilde{M}_n(t)=\int_0^t \tilde{b}_n(t-y) dm(y)$ for all $t\ge 0$, which is equivalent to \eqref{Eq36} and \eqref{Eq36ni}, up to multiplication by $ e^{\eta_n \delta t}$. However, as pointed out in \cite{W15}, this solution is hardly explicit in practice because $\tilde{b}_n(.)$ depends on $\tilde{M}_\ell(.)$, $\ell<n$. Only when $n=n(i)$, we find a simple expression which was also considered in \cite[Example 3]{W15} as $k=1$ and $n_1=1$. In this case, one finds (\ref{exp_tilde_b}) given by
\begin{equation}\label{bnt}
\tilde{b}_{n(i)}(t)=\nbE[X_i].\int^t_0 e^{-\delta y}\overline{\omega}_{\delta,i}(t-y)dF(y).
\end{equation}
So, in general, the expression for $\tilde{M}_n(t)$ at time $t$ depends on the whole trajectory of $\tilde{M}_\ell(y)$, $\ell<n$, for $y\in [0,t]$ as it is also obvious from (\ref{Eq36}). Furthermore, the renewal function $t\mapsto m(t)$ is not always explicit.
\begin{cor}
The mgf $\tilde{\psi}(s,t)$ of $\tilde{Z}(t)$ satisfies the integral-renewal equation
\begin{equation}\label{integral_eq_psi_tilda}
\tilde{\psi}(s,t)=\overline{F}(t)+\int_0^t M^*_{t-y,X}(e^{\delta(t-y)}s)\tilde{\psi}(s,t-y)dF(y),\quad t\ge 0,
\end{equation}
for all $s\in\nbR^k$.
\end{cor}
\begin{proof}
The renewal equation \eqref{integral_eq_psi_tilda} is obtained thanks to relation $\tilde{\psi}(s,t)=\psi(e^{\delta t}s,t)$ and by using \eqref{expression_psi} as well as a classical renewal argument.
\end{proof}

The above corollary is useful to find a closed form expression for $\tilde{\psi}(s,t)$ when arrivals occur according to a Poisson process.
\begin{prop}[Poisson arrival and general delay]\label{prop_case_Poisson}
If $\tau_1\sim {\cal E}(\lambda)$ then one has the following expression
\begin{equation}\label{transient_Poisson}
\tilde{\psi}(s,t)=\exp\left[ \lambda \int_0^t \left( M^*_{v,X}(e^{\delta v}s)-1\right) dv\right],\quad t\ge 0,\ s\in \nbR^k.
\end{equation}
Then, the mgf of $Z(t)$ is obtained explicitly by $\psi(s,t)=\tilde{\psi}(e^{-\delta t}s,t)$.
\end{prop}
\begin{proof}
When $\tau_1\sim {\cal E}(\lambda)$, renewal equation \eqref{integral_eq_psi_tilda} leads, up to a change of variable $y:=t-y$ in the integral, to
$$
\tilde{\psi}(s,t)=e^{-\lambda t}+\int_0^t M^*_{y,X}(e^{\delta y}s)\tilde{\psi}(s,y)\lambda e^{-\lambda (t-y)} dy,\quad t\ge 0,
$$
which, derived with respect to $t$, yields the linear differential equation
$$
\partial_t \tilde{\psi}(s,t)=\lambda\left[ -1+ M^*_{t,X}(e^{\delta t}s)\right] \tilde{\psi}(s,t)
$$
of which solution is given by \eqref{transient_Poisson}. Note that the above differential equation is also available in a similar form in \cite[Theorem 3.1]{MT02}.
\end{proof}
Two remarks are to be deduced from Proposition \ref{prop_case_Poisson}. First, since the pdf of the exponential distribution is upper bounded, Condition {\bf (A1)} is fulfilled, and thus one has from upcoming Theorem \ref{theo_conv_distrib} in Section \ref{sec_conv_moments_distribution}, that $\tilde{Z}(t)$ converges in distribution towards some light tailed random vector ${\cal Z}_\infty$. Thus, it is immediate from \eqref{transient_Poisson} that the mgf of ${\cal Z}_\infty$ is, when $\tau_1\sim {\cal E}(\lambda)$, given by
$$
\nbE\Big[e^{<s,{\cal Z}_\infty>}\Big]=\lim_{t\to \infty} \tilde{\psi}(s,t)=\exp\left[ \lambda \int_0^\infty \left( M^*_{v,X}(e^{\delta v}s)-1\right) dv\right],\quad s\in \nbR^k .
$$
Second, one is able to recover some well known result in the $M/G/\infty$ queue by setting $\delta=0$. For example, when $k=1$ and $X=X_1$, one computes that $M^*_{t,X}(s)=1+(e^s-1)\overline{W}(t)$, and \eqref{transient_Poisson} reduces to
$$
\tilde{\psi}(s,t) =\psi(s,t)= \exp\left[ \lambda \int_0^t \overline{W}(v)dv . [e^s-1]\right]
$$
recovering that the distribution of the number of customers in an infinite server queue with Poisson arrivals of intensity $\lambda$ is Poisson distributed with parameter $\lambda \int_0^t \overline{W}(v)dv$ at time $t$, see \cite[Theorem 1, p.160]{T62}. When $\delta=0$, \eqref{transient_Poisson} in Proposition \ref{prop_case_Poisson} is similar to the results obtained in Section 3.1 of \cite{LKT90}, concerning infinite server queues with Poisson arrivals. 

\section{General results: Convergence of joint moments and distribution}\label{sec_conv_moments_distribution}
We are interested in the limiting behaviour of the process $\tilde{Z}(t)$ when arrivals and delays have a general distribution. It may be difficult to compute its distribution in all generality, however some information may be obtained if we add a specific assumption on the arrival process $\{ N_t\}_{t\ge 0}$. Our first immediate result is convergence of joint moments of $\tilde{Z}(t)$:
\begin{prop}\label{prop_asymptotics}
One finds the following asymptotic result for the joint moments of $Z(t)$, for all $n\in\nbN^k$:
\begin{equation*}
\lim_{t\to \infty} \tilde{M}_n(t)=\chi_n\quad \iff\quad M_n(t)\sim \chi_n e^{-\eta_n \delta t},\qquad t\to\infty,
\end{equation*}
where
\begin{equation}
0< \chi_n:=\frac{\displaystyle\int_0^\infty \tilde{b}_n(t) dt}{\displaystyle\nbE[\tau_1]}<+\infty ,
\label{asymp_tilde_M}
\end{equation}
and $\tilde{b}_n(t)$ is given by (\ref{exp_tilde_b}).
\end{prop}
\begin{proof} See Section \ref{proof_prop_asymptotics}.
\end{proof}

A direct consequence of Proposition \ref{prop_asymptotics} when $n=n(i)$ with \eqref{bnt} yields the result for the first moment in the following corollary.

\begin{cor}[First marginal moment: Arbitrary time delays]\label{Coro_single_type_of_input}
When $n=n(i)$, the mean of $\tilde{Z}_i(t)$ in \eqref{def_Z_tilde} with \textbf{arbitrary} time lag distribution $L_i$ is asymptotically obtained as
\[
\lim_{t\to\infty} \nbE[\tilde{Z}_{n(i)}(t)]=\chi_{n(i)}, 
\]
where
\begin{equation}\label{single_type_input_chi1}
\chi_{n(i)}=\frac{\nbE[X_i]\nbE[L_{i}]\tilde{w}_{1,i}(\delta)}{\nbE[\tau_1]},
\end{equation}
and $\tilde{w}_{1,i}(\delta)=\int_0^\infty e^{-\delta x}\overline{W}_i(x)dx/\nbE[L_{i}]$. This is a generalization of Corollary 3 in \cite{W15} in which it is assumed that $X_i= 1$ and $\delta=0$.
\end{cor}
\begin{remark}[Little's law revisited]\label{Little_revisit}
{\normalfont Expression (\ref{single_type_input_chi1}) gives an interesting interpretation in a queueing context. Let us suppose here (without loss of generality) that $X_i=1$ (i.e. customers do not arrive in batches). Then (\ref{single_type_input_chi1}) leads to
\begin{equation}\label{Litte_bis}
\lim_{t\to\infty} \nbE[\tilde{Z}_{n(i)}(t)]=\chi_{n(i)}=\frac{\nbE[L_{i}]\tilde{w}_{1,i}(\delta)}{\nbE[\tau_1]}.
\end{equation}
When $\delta=0$, $\tilde{Z}_{n(i)}(t)$ is the number of customers at time $t$ in infinite server queues; In the case of $\tilde{w}_{1,i}(\delta)=1$, (\ref{Litte_bis}) is just a rephrasing of Little's law which says that the limiting expected number of customers in the queue is equal to the arrival rate multiplied by the mean service time. When $\delta>0$, we notice that $\nbE[L_{i}]\tilde{w}_{1,i}(\delta)=\nbP(L_i>E_\delta)/\delta$ where $E_\delta\sim {\cal E}(\delta)$ is a rv which is independent from everything, so that (\ref{Litte_bis}) leads to
\begin{equation}\label{Little_generalized}
\lim_{t\to\infty} \nbE[\tilde{Z}_{n(i)}(t)]=\frac{1}{\nbE[\tau_1]}\frac{\nbP(L_{i}>E_\delta)}{\delta}=\frac{1}{\nbE[\tau_1]}\nbP(L_{i}>E_\delta) \nbE[E_\delta].
\end{equation} 
The asymptotic expression in (\ref{Little_generalized}) implies that the limiting expected number of customers of which residual service time is no more than horizon $E_\delta\sim {\cal E}(\delta)$ is equal to the arrival rate multiplied by the expected horizon time, and the proportion of customers of which service time did exceed this horizon $E_\delta$. So, (\ref{Little_generalized}) can be regarded as a generalization of Little's Law in the $G/G/\infty$ context.}
\end{remark}

We note that in Proposition \ref{prop_asymptotics} coefficients $\chi_n$, $n\in\nbN^k$ are in general not directly available, as the function $t\mapsto \tilde{b}_n(t)$ in the integral (\ref{asymp_tilde_M}) does not have an easy expression, and are defined recursively in the function of $ t\mapsto \tilde{M}_\ell(t)$, $\ell<n$. We thus provide in the following easily computable bounds for the $\chi_n$'s and a uniform upper bound in $t$ for $\tilde{M}_n(t)$ if we impose that {\bf (A1)} holds.
\begin{prop}[Upper bounds for the joint moments]\label{tildeMnRn}
Let us suppose that {\bf (A1)} holds. One has the following bounds for all $n\in\nbN^k$:
\begin{eqnarray}
\chi_n & \le & \frac{1}{ \nbE(\tau_1)}\,R_n,\label{bound_Cn}\\
\tilde{M}_n(t)& \le & C R_n,\quad \forall t\ge 0, \label{bound_M_tilde}
\end{eqnarray}
where $(R_n)_{n\in\nbN^k}$ is defined recursively by
\begin{equation}
\begin{array}{rcl}
R_{n(i)}&=&  \nbE[X_i]
\delta^{-1}
\left\{1-\nbE\left[ e^{-\delta L_i}\right]\right\},
\quad  i=1,\ldots,k,\\
\label{rec_Rn}
	R_n&=& \displaystyle  \sum_{ \ell < n} {{n_1}\choose{\ell_1}}\cdots{{n_k}\choose{\ell_k}} \nbE\bigg[ \prod_{j=1}^k X_j^{n_j-\ell_j}\bigg] \max_{i\in C_{\ell ,n}} \nbE[L_i].\, R_\ell,\quad n\in\nbN^k\!\setminus \!\{n(i),\ i=1,\ldots,k\}.
\end{array}
\end{equation}
 Here, the constant $C$ is the upper bound for renewal density $u(t)$ in Lemma \ref{lemma_density_upper_bound}.
\end{prop}
\begin{proof} See Section \ref{proof_tildeMnRn}.
\end{proof}
We remark that (\ref{bound_Cn}) provides a simple bound for the limiting joint moments which is immediately computable, and  (\ref{bound_M_tilde}) gives some information on the transient joint moments (i.e. on the whole trajectory $ t\mapsto \tilde{M}_n(t)$). It is often more complicated to compute (\ref{bound_M_tilde}), as $C$ is not always explicit (as explained in Remark \ref{rem_constant_C}). Also, the existence of the upper bound $C$ is proved thanks to the fact that $\lim_{t\to \infty} u(t)=1/\nbE(\tau_1)$, as shown in the proof of Lemma \ref{lemma_density_upper_bound} at the beginning of Section \ref{sec:Proofs}, which implies that $1/\nbE(\tau_1)\le C$. Hence (\ref{bound_Cn}) is tighter than (\ref{bound_M_tilde}).

In the following example, we calculate higher moments of two types of discounted IBNR claims until time $t$ under the same model setting as \cite[Section 4]{W15} for comparison purposes.

\begin{example}[Two types of inputs, Erlang(2) arrival process]\normalfont\label{Erlang2Two}
Suppose that there are two types of claim amounts distributed as the bivariate gamma proposed by \cite{Izawa} with the parameters $\alpha=2$, $\beta_1=1$, $\beta_2=0.5$, and $\rho=0.5$. For the time delay distributions, $W_1(t)=1-e^{-t}$ and $W_2(t)=1-e^{-5t}$. We consider Erlang(2) process for claim counting process with $f(t)=t e^{-t}$. In addition, the discounting factor $\delta$ is assumed to be 5\%. In this case, the renewal density $u(t)$ in (\ref{Expression_density_renewal}) is $0.5-0.5e^{-2x}$ and thus, we set $C$ in (\ref{redenupper}) as $0.5$. Then from Proposition \ref{tildeMnRn}, the bounds for first two moments and the joint expectation, i.e. $\tilde{M}_n(t)=e^{\eta_n \delta t}\nbE[Z_1^{n_1}(t)Z_2^{n_2}(t)]$ for $n \in \{(1,0),(0,1),(2,0),(0,2),(1,1)\}$, are first calculated and compared with the exact results obtained from the expression given by \cite{W15}.
\begin{table}[htbp]
  \centering
    \begin{tabular}{|l|cc|cc|c|}
				\hline
    $t$ & $e^{\delta t}\nbE[Z_1 (t)]$   & $e^{\delta t}\nbE[Z_2 (t)]$    & $e^{2\delta t}\nbE[Z_1^2 (t)]$   & $e^{2\delta t}\nbE[Z_2^2 (t)]$  & $e^{2\delta t}\nbE[Z_1(t)Z_2 (t)]$ \\ \hline
 1 & 0.3806 & 0.7712 & 1.1565 & 11.6324 & 1.6914   \\
 5 & 0.9396 & 0.9900 & 3.2718 & 14.9848 & 2.5205   \\
 100 & 0.9524 & 0.9901 & 3.3320 & 14.9860 &  2.5333    \\
 1000& 0.9524 & 0.9901 & 3.3320 & 14.9860 & 2.5333 \\
 \hline\hline
Bound & 0.9524 & 0.9901 &  4.9048 & 16.9802 & 4.0258 \\
				\hline
     \end{tabular}
  \caption{Exact values and bound for the first two moments and joint moment of $\tilde{M}_n(t)$}\label{Ex8}
\end{table}
The results are summarized in Table \ref{Ex8}. It is worth noting that it is obviously simpler to use $\max_{i\in C_{\ell ,n}} \nbE[L_i]$ in (\ref{rec_Rn}). However, a closer look at the proof of Proposition \ref{tildeMnRn} leading to \eqref{rec_Rn} reveals that  $\max_{i\in C_{\ell ,n}} \nbE[L_i]$ can be replaced by $ \prod_{j\in C_{\ell ,n}} \nbE[L_j]^{1/|C_{\ell ,n}|}$ or  $\int_0^\infty \prod_{j\in C_{\ell ,n}}  \overline{W}_j(t)  \, dt$, the latter yielding tighter bounds, which are the ones displayed in Table \ref{Ex8}. In this example, this quantity is straightforward to calculate, hence we have utilized this integral expression to calculate $R_n$ in (\ref{rec_Rn}). In addition, it turns out from (\ref{Eq36}) that the the expression for the $m$th moment even for each type of claim (i.e. $n=(0,m)$ or $n=(m,0)$) is not efficient in the computational point of view since the higher moment requires an integration of the analytic expression of the lower moment. On the other hand, (\ref{rec_Rn}) is only a simple finite sum which is simplified for this case as
\[
	R_n= \displaystyle  \nbE[L_{i}] \sum^{n_i-1}_{\ell_i=0} {{n_i}\choose{\ell_i}} \nbE\big[ X_i^{n_i-\ell_i}\big]  .\, R_\ell,\qquad  n\in\nbN^2\!\setminus \!\{n(i),\ i=1,2 \}, 
\]
starting with $R_{(0,0)}=1$ and using (\ref{rec_Rn}) when $n=n(i)$ (i.e. $R_{(1,0)}$, $R_{(0,1)}$). For example, for $m=3,4,5$ and type-1 claim, it is immediately obtainable as $R_{(3,0)}=35.29$, $R_{(4,0)}=335.14$, and $R_{(5,0)}=3968.57$.
\end{example}

Proposition \ref{tildeMnRn} is useful for two reasons. First,  as illustrated in the previous example, we remark that coefficients $R_n$, $n\in \nbN^k$, in (\ref{rec_Rn}) can be easily computed because $R_n$ is a linear function of the $R_\l$, $\l <n$, and only involves the joint moments of $X=(X_1,\ldots,X_k)$, the Laplace transform of the $L_1$,\ldots,$L_k$  as well as their expectations. So, simple bounds are available, which is useful since it is not possible in general to compute the distribution (or even moments) of the process. Second, Proposition \ref{prop_asymptotics} leads to $\tilde{M}_n(t)$ converging towards $\chi_n$. Since $\tilde{M}_n(t)$ is the joint moments of $\nbR^k$ valued process $\{\tilde{Z}(t)\}_{t\ge 0}$, this suggests in turn that this process converges in distribution.
As convergence of moments does not always implies convergence in distribution, we give some sufficient conditions such that this latter holds, and we prove it thanks to the bounds obtained in Proposition \ref{tildeMnRn}. In the following we address the limiting behaviour in distribution of process $\{\tilde{Z}(t)\}_{t\ge 0}$ under {\bf (A1)} and {\bf (A2)}.

\begin{theorem}\label{theo_conv_distrib}
Let us suppose that {\bf (A1)} and {\bf (A2)} hold.
Then one has the result of convergence in distribution for $\tilde{Z}(t)$:
\begin{equation*}\label{conv_distribution}
\tilde{Z}(t)=e^{\delta t}Z(t)\stackrel{\cal D}{\longrightarrow} {\cal Z}_\infty,\quad t\to \infty ,
\end{equation*}
where ${\cal Z}_\infty=({\cal Z}_{\infty, 1},\ldots,{\cal Z}_{\infty,k})={\cal Z}_\infty(\delta)$ is a vector of light tailed random variables with the joint moments
\[
\nbE\bigg[ \prod_{i=1}^k {\cal Z}_{\infty, i}^{n_i}\bigg]=\chi_n=\chi_n(\delta)
\]
given by (\ref{asymp_tilde_M}) for $n\in \nbN^k$.
\end{theorem}
\begin{proof} See Section \ref{proof_theo_conv_distrib}.
\end{proof}

\section{Joint moments with exponential delays}\label{sec_expo}
Let us note that Theorem \ref{theo_conv_distrib} holds for general interarrival times $\tau_i$ that satisfy {\bf (A1)}, and general time delays $L_j$'s. The aim of this subsection is to prove that the $\chi_n$'s are explicit when the $L_j$'s are exponentially distributed. We suppose for simplicity that all $L_j$'s for $j=1,\ldots,k$, have the same distribution ${\cal E}(\mu)$, for some $\mu>0$. Note that we may obtain similar results as will be given in the following for more general cases such as a mixture or a combination of exponential distributions, but the expressions would only be more complicated.

For notational convenience, let $\L^M_n(u)$ and $\L^b_n(u)$ for $u\ge 0$ and $n\in\nbN^k$, be the Laplace transforms of $\tilde{M}_n(\cdot)$ and $\tilde{b}_n(\cdot)$ respectively
$$
\L^M_n(u):= \int_0^\infty e^{-uy}\tilde{M}_n(y)dy,\quad \L^b_n(u):= \int_0^\infty e^{-uy}\tilde{b}_n(y)dy.
$$
Note that these Laplace transforms exist (i.e. the integrals converge) respectively when $u>0$ and $u\ge 0$ since $\tilde{M}_n(y)$ converges to some finite limit $\chi_n$ as $y\to \infty$, and $\tilde{b}_n(\cdot)$ is integrable (as proved in Proposition \ref{prop_asymptotics}). The following lemma gives a recursive expression of $\L^b_n(u)$. 
\begin{lemm}\label{lemma_recursive}
When time delays $L_j$'s are ${\cal E}(\mu)$ distributed, the Laplace transform of $\tilde{b}_n(\cdot)$ in (\ref{exp_tilde_b}) is obtained as
\begin{equation}\label{recursion_LT_b1}
\L^b_{n(i)}(u)=\nbE[X_i]\frac{\mu}{(\mu+\delta)(\mu+u)} \, \L^\tau (u) ,\qquad i=1,\ldots,k,
\end{equation}
and
\begin{multline}\label{recursion_LT_b}
\L^b_n(u)= B_{\mathbf{0}, n}  \frac{\L^\tau (u)}{u+ |C_{\mathbf{0},n}|\mu} +\sum_{\mathbf{0}<\l < n} B_{\l, n} \frac{\L^\tau (u)}{1-\L^\tau(u+|C_{\l ,n}|\mu)}\, \L^b_\l(u+|C_{\l ,n}|\mu),\\
n\in\nbN^k\!\setminus \!\{n(i),\ i=1,\ldots,k\},
\end{multline}
where $\mathbf{0}$ is a zero vector in $\nbN^k$,
\begin{equation}\label{Bln}
B_{\l, n}:={{n_1}\choose{\ell_1}}\cdots{{n_k}\choose{\ell_k}}
\nbE\bigg[ \prod_{j=1}^k X_j^{n_j-\ell_j}\bigg] \prod_{j\in C_{\l ,n}} \frac{\mu}{\mu+(n_j-\l_j)\delta},
\end{equation}
and we recall that $C_{\l,n}=\{ j=1,\ldots,k|\ \ell_j<n_j\}\subset \{1,\ldots,k \}$.
\end{lemm}
\begin{proof} See Section \ref{proof_lemma_recursive}.
\end{proof}
The following theorem shows that the $\chi_n$'s can be computed as a function of coefficients $D_n(j) = \L^b_n(j\mu) $ which are defined recursively.
\begin{theorem}\label{expser}
Let us denote $D_n(j):=\L^b_n(j\mu)$ for $j\in\nbN$ and $n\in \nbN^k$. When time delays $L_j$'s are ${\cal E}(\mu)$ distributed, the joint moments $\chi_n$ for $n\in\nbN^k$ of $\mathcal{Z}_\infty=\mathcal{Z}_\infty(\delta)$ (the limiting distribution of $e^{\delta t}Z(t)$), are given by
\begin{eqnarray}
\chi_{n(i)}&=&\dfrac{\nbE[X_i]}{\nbE[\tau_1]}\bigg(\dfrac{1}{\mu+\delta}\bigg),
\qquad i=1,\ldots,k,
\label{C_1_expo}\\
 \chi_n&= &\dfrac{1}{\nbE[\tau_1]}\left(B_{\mathbf{0}, n}  \frac{1}{|C_{\mathbf{0},n}|\mu} + \sum_{\mathbf{0}<\l < n} B_{\l, n} \frac{1}{1-\L^\tau(|C_{\l ,n}|\mu)}\,D_\l(|C_{\l ,n}|)\right),\nonumber\\
&& n\in\nbN^k\!\setminus \!\{n(i),\ i=1,\ldots,k\},
\label{C_n_expo}
\end{eqnarray}
where $D_n(j)$'s for $j\in\nbN$ and $n\in\nbN^k$ are obtained recursively as:
\begin{eqnarray}
D_{n(i)}(j)&=& \nbE[X_i]\frac{\mu}{(\mu+\delta)([j+1]\mu)} \,\L^\tau (j\mu),\quad i=1,\ldots,k,
\label{D_n_i,expression}\\
D_n(j)&=& B_{\mathbf{0}, n}  \frac{\L^\tau (j\mu)}{[j+|C_{\mathbf{0},n}|]\mu} +\sum_{\mathbf{0}<\l < n} B_{\l, n} \frac{\L^\tau (j\mu)}{1-\L^\tau([j+|C_{\l ,n}|]\mu)}\,D_\l([j+|C_{\l ,n}|]),\nonumber\\
&& n\in\nbN^k\!\setminus \!\{n(i),\ i=1,\ldots,k\},
\label{D_n,expression}
\end{eqnarray}
with $B_{\l,n}$ in (\ref{Bln}).
\end{theorem}
\begin{proof}
From (\ref{asymp_tilde_M}), using (\ref{recursion_LT_b1}) and (\ref{recursion_LT_b}) when $u=0$, we find
(\ref{C_1_expo}) and (\ref{C_n_expo}) respectively.
In addition, (\ref{D_n_i,expression}) and (\ref{D_n,expression}) are obtainable by setting $u=j\mu$ in (\ref{recursion_LT_b1})
and (\ref{recursion_LT_b}) respectively.
\end{proof}
We remark that a close look at (\ref{C_n_expo}) and (\ref{D_n,expression}) reveals that computation of the {\it infinite} sequences $(D_\l(j))_{j\in\nbN}$ for all $\l<n$ is not needed to obtain $\chi_n$. Since $|C_{\l ,n}|$ is bounded by $k$, it is not hard to see that one needs to compute (recursively) $D_\l(j)$ for $\l<n$ and for $j\le k\eta_n$ (i.e. only for a finite number of $j$'s). Moreover, the values of those $D_n(j)$'s may be stored in memory while computing the successive $\chi_n$ as $\eta_n$ increases, and thus one does not need to  recompute them each time. Hence the algorithm (\ref{C_n_expo}) is relatively not too costly, which is numerically illustrated in the following example.

\begin{example}[Example \ref{Erlang2Two} revisited]\normalfont\label{Erlang2Two1} The same model for the discounted total IBNR processes is assumed except that the time delay distribution for both claims is $W(t)=1-e^{-5t}$, i.e. $L_1$ and $L_2$ are both ${\cal E}(5)$ distributed. Then we calculate the bounds and the exact values for $\chi_n$ which is the asymptotic value for $\tilde{M}_n(t)$ in Proposition \ref{prop_asymptotics}. For $n=(n_1,n_2)$ with $n_1\leq 2$ and $n_2\leq 2$, bounds are calculated from Proposition \ref{tildeMnRn} while the exact values are computed from Theorem \ref{expser}.
All results for $\tilde{M}_n(t)=e^{\eta_n \delta t}\nbE[Z_1^{n_1}(t)Z_2^{n_2}(t)]$ for $n \in \{(1,0),(0,1),(2,0),(0,2),(1,1),(2,1),(1,2),(2,2)\}$ are summarized in Table \ref{Ex12}. It displays that bounds are easily computable and the results are quite close to the exact asymptotic values for all orders we considered.
\begin{table}[htbp!]
  \centering
    \begin{tabular}{|l|cccc|}
				\hline
    $ $ & $e^{\delta t}\nbE[Z_1 (t)]$   & $e^{\delta t}\nbE[Z_2 (t)]$    & $e^{2\delta t}\nbE[Z_1^2 (t)]$   & $e^{2\delta t}\nbE[Z_2^2 (t)]$    \\ \hline
 Exact & 0.1980  & 0.9901  & 0.5994  & 14.9860    \\
 Bound & 0.1980 & 0.9901 & 0.6792 & 16.9802   \\
				\hline
				    $ $ & $e^{2\delta t}\nbE[Z_1(t)Z_2 (t)]$   & $e^{3\delta t}\nbE[Z_1(t)^2 Z_2 (t)]$    & $e^{3\delta t}\nbE[Z_1 (t) Z_2^2(t)]$   & $e^{4\delta t}\nbE[Z_1^2(t)Z_2^2 (t)]$    \\ \hline
 Exact & 1.2814  & 5.9553  & 29.7765 &  153.9510 \\
 Bound & 1.6460 & 6.9267 &  34.6337 & 167.8850   \\
				\hline
     \end{tabular}
     \caption{Exact values and bounds for asymptotics of $\tilde{M}_n(t)$}\label{Ex12}
\end{table}

\end{example}

Some special cases such as the higher marginal moments and the joint mean and covariance are given in the following.
\begin{cor}[Higher marginal moments: Exponential time delays]\label{nmoment1} The $r$th  marginal moment of  $Z_i(t)$ in (\ref{Zjt}) with exponential time delays is asymptotically obtained as
\[
\nbE[Z_i^r(t)] \sim \chi_{r.n(i)}\, e^{- r\delta t},\qquad t\to\infty,\qquad r\in \nbN,
\]
where
\begin{equation}\label{CHIk11}
\chi_{n(i)}=\frac{\nbE[X_i]}{\nbE[\tau_1]} \bigg(\frac{1}{\mu+\delta}\bigg)\qquad \mbox{with }r=1,
\end{equation}
and
\begin{equation}\label{CHIk1n}
\chi_{r.n(i)}= \dfrac{1}{\nbE[\tau_1]}\left(\nbE[X_i^r] \frac{1}{\mu+
r\delta}  + \sum_{\textcolor{black}{l}=1}^{r-1} {{r}\choose{\textcolor{black}{l}}}
\nbE\big[  X_i^{r-\textcolor{black}{l}}\big]  \frac{\mu}{\mu+
(r-\textcolor{black}{l})\delta} \frac{D_{\textcolor{black}{l.n(i)}}(1)}{1-\L^\tau(\mu)}\right),\qquad r=2,3,\ldots,
\end{equation}
and $D_{\textcolor{black}{l.n(i)}}(1)$ recursively available from the formulas (\ref{D_n_i,expression}) with $\textcolor{black}{l}=1$ and (\ref{D_n,expression}) where $n=\textcolor{black}{l}.n(i)$,
and
\begin{multline*}
D_{\textcolor{black}{l.n(i)}}(j)=\nbE[X_i^{\textcolor{black}{l}}] \frac{\mu}{\mu+
n\delta}\frac{\L^\tau(j\mu)}{[j+1]\mu}\\
+\sum_{\textcolor{black}{l'}=1}^{\textcolor{black}{l}-1} {{n}\choose{\textcolor{black}{l'}}}
\nbE\big[  X_i^{\textcolor{black}{l-l'}}\big]  \frac{\mu}{\mu+
(\textcolor{black}{l-l'} )\delta}\frac{\L^\tau(j\mu)}{1-\L^\tau([j+1]\mu)} D_{\textcolor{black}{l'.n(i)}}(j+1).
\end{multline*}
\end{cor}
\begin{proof}
Let $i\in\{ 1,...,k\}$ and $r\in \nbN^*$. The result follows by using Theorem \ref{expser} with $ n=r .n(i)$, which is such that $n_j= r\delta_{i,j}$, $j=1,...,k$. Note that in that case, the sum over $0<\ell<n$ in \eqref{C_n_expo} and (\ref{D_n,expression}) is necessarily such that $\ell = l .n(i)$ for $ l=1,...,r-1$, and that $|C_{\l ,n}|=1$.
\end{proof}
It is noted that the form given in Theorem 3 of \cite{W15} was not suitable to derive the asymptotic behavior of $Z_i(t)$. It reveals only that this quantity is asymptotically close to zero. Hence Corollary \ref{nmoment1} is useful for calculating higher moments of $Z_i(t)$ in any order for a large $t$ when time delays are exponentially distributed.

\begin{remark}
{\normalfont When $\delta=0$ and $X_i=1$, the model in Corollary \ref{nmoment1} reduces to the classical $G/M/\infty$ queue, which was extensively studied by Tak\'acs \cite[Chapter 3, Section 3]{T62}. More precisely, the results in this corollary are comparable to \cite[Theorem 2, p.166]{T62}, \cite[Theorem 2]{NC72} and \cite[Corollary of Theorem 1]{S72} where the approach is different and the distribution of the asymptotic queue level is derived but it is in the form of an infinite sum involving so-called binomial moments.}
\end{remark}
Next, we compute the covariance of $Z_1(t)$ and $Z_2(t)$ when $k=2$. We thus let $n=(n_1,n_2)=(1,1)$ (i.e. $\l=(\l_1,\l_2) \in \{(0,0),(0,1),(1,0)\}$). From (\ref{exp_tilde_b}) and (\ref{phi_l}), we have
\begin{align}
\tilde{b}_n(t) &=\sum_{\l_1,\l_2 \setminus(\l_1, \l_2)<(n_1,n_2)} {{n_1}\choose{\ell_1}}{{n_2}\choose{\ell_2}}
\nbE\bigg[ \prod_{j=1}^2 X_j^{n_j-\ell_j}\bigg]\ph_{\ell ,n}(t)\nonumber\\
&= \nbE [X_1X_2] \ph_{(0,0),n}(t) + \nbE[X_1] \ph_{(0,1),n}(t) + \nbE[X_2] \ph_{(1,0),n}(t),
\label{exp_tilde_b2}
\end{align}
where $\ph_{(0,0),n}(t)= \nbE\big[ e^{ 2\delta (t-\tau_1)} \overline{\omega}_{\delta,1}(t-\tau_1)\overline{\omega}_{\delta,2}(t-\tau_1)
.\nbu_{\{\tau_1<t\}}\big]$ because of $\tilde{M}_{(0,0)}(t-\tau_1)=1$),
$\ph_{(0,1),n}(t)= \nbE\big[ e^{ \delta (t-\tau_1)}\tilde{M}_{(0,1)}(t-\tau_1) \overline{\omega}_{\delta,1}(t-\tau_1)
.\nbu_{\{\tau_1<t\}}\big]$, and
$\ph_{(1,0),n}(t)= \nbE\big[ e^{ \delta (t-\tau_1)}\tilde{M}_{(1,0)}(t-\tau_1) \overline{\omega}_{\delta,2}(t-\tau_1)
.\nbu_{\{\tau_1<t\}}\big]$. As shown previously, (\ref{exp_tilde_b2}) is simplified when $L_i$ for $i=1,2$, is exponentially distributed. In this case, the joint expectation and the covariance of two types of inputs are presented in the following.
\begin{cor}[Joint mean and covariance: Exponential time delays]
The joint mean of two types of $Z_1(t)$ and $Z_2(t)$ in (\ref{Zjt}) where the time delay of type-$1$ and type-$2$ inputs, $L_1$ and $L_2$ are ${\cal E}(\mu)$ distributed, is asymptotically given by
\[
\nbE[Z_1(t)Z_2(t)] \sim \chi_{(1,1)} e^{-2\delta t},\qquad t\to\infty,
\]
where 
\begin{equation}\label{CHI11}
\chi_{(1,1)}=\frac{1}{\nbE[\tau_1]}\frac{\mu}{(\mu+\delta)^2} \bigg[\frac{\nbE[X_1X_2]}{2}+ \nbE[X_1]\nbE[X_2]\frac{\L^\tau(\mu)}{1-\L^\tau(\mu)} \bigg].
\end{equation}
Consequently, the covariance is given by
\[
\Cov[Z_1(t),Z_2(t)] \sim \xi_{(1,1)} e^{-2\delta t},\qquad t\to\infty,
\]
where $\xi_{(1,1)}=\chi_{(1,1)}-\frac{\nbE[X_1]\nbE[X_2]}{\nbE[\tau_1]^2(\mu+\delta)^2}$ with $\chi_{(1,1)}$ given in (\ref{CHI11}).
\end{cor}
\begin{proof}
From Theorem \ref{expser} when $n=(n_1,n_2)=(1,1)$  (i.e $|C_\l|=1$ when $\l=(\l_1,\l_2) \in \{(1,0),(0,1)\}$), we have
\begin{equation}\label{chi11}
\chi_{(1,1)}=\frac{1}{\nbE[\tau_1]}\bigg[B_{(0,0),(1,1)}\frac{1}{2\mu}+B_{(1,0),(1,1)}\frac{D_{(1,0)}(1)}{1-\L^\tau(\mu)}+B_{(0,1),(1,1)}\frac{D_{(0,1)}(1)}{1-\L^\tau(\mu)}\bigg].
\end{equation}
But from (\ref{Bln}), $B$'s are given by
\[
B_{(0,0),(1,1)}=\nbE[X_1 X_2] \bigg(\frac{\mu}{\mu+\delta}\bigg)^2, ~~B_{(1,0),(1,1)}=\nbE[X_2] \frac{\mu}{\mu+\delta},~~B_{(0,1),(1,1)}=\nbE[X_1] \frac{\mu}{\mu+\delta}.
\]
Also, $D_{n(i)}(1)$ for $i=1,2$ is available from (\ref{D_n_i,expression}) as
$
D_{n(i)}(1)=\nbE[X_i]\frac{\mu}{(\mu+\delta)(2\mu)} \,\L^\tau (\mu).
$
Combining results given above, (\ref{CHI11}) is expressed thanks to (\ref{chi11}) and the asymptotics of $\nbE[Z_1(t)]$ and $\nbE[Z_2(t)]$ obtained in Corollary \ref{nmoment1}.
\end{proof}

An interesting consequence of Lemma \ref{lemma_recursive} is that the expression of $\L^M_n(u)$ for all $u>0$ can be obtained recursively thanks to the relation 
\begin{equation}\label{LT_M_b}
\L^M_n(u)=\dfrac{\L^b_n(u)}{1- \L^\tau (u)},\quad \forall u> 0,\quad n\in\nbN^k\!\setminus \!\{n(i),\ i=1,\ldots,k\},
\end{equation}
(which stems from the renewal equation \eqref{renewal_M_tilde}) as well as relations \eqref{recursion_LT_b1} and \eqref{recursion_LT_b}. As in the computation of $(D_\l(j))_{j\in\nbN}$, this requires computing $\L^b_\l(u+j\mu)$ for a finite number of $j$'s and $\l <n$ only. This enables us to obtain the Laplace transform in $y$ of $\tilde{\psi}(s,y)$ (defined in \eqref{mgf_Z_tilde}), the mgf of $\tilde{Z}(y)$, thanks to the formula
$$
\int_0^\infty e^{-uy} \tilde{\psi}(s,y)dy = \sum_{n\in \nbN^k}\prod_{i=1}^k \frac{s_i^{n_i}}{n_i!} \L^M_n(u),\quad u>0,\quad s\in\nbR^k,
$$
and gives some information on the transient behaviour of $\tilde{Z}(t)$, see \cite[Theorem 3, p.168]{T62} (which deals with the case of $k=1$ in the current model), for a comparable result.

\section{Single input with exponential delays}\label{sec_single_queue}
We further narrow down the scope of Section \ref{sec_expo} for exploration of the particular case where delays are exponentially distributed and $k=1$. As we deal with a one dimensional process, we drop a subscript $j$ in $L_{i,j}$ 
which represents the service time for the $j$-type of input (i.e. write $L_i$ for $i\in\nbN$), and denote by $L$ for the generic service time. Similarly, we write $X$ instead of $X_1$, $W(t)$ for $W_1(t)$, and the $r$th limiting moment of $\tilde{Z}(t)$ is written $\chi_r$ instead of $\chi_{r.n(i)}$. The first subsection gives some information on the rate of convergence of the first moment of $\tilde{Z}(t)$. As a result, some limiting behaviour of the workload in infinite server queues is studied.

\subsection{High order expansions}\label{sec:high}

We study in this subsection how fast the first moment $\tilde{M}_1(t)=\nbE[e^{\delta t}Z(t)]$ converges to $\chi_1$ given in Proposition
\ref{prop_asymptotics} when $t\rightarrow \infty$.
As $\tilde{M}_1(t)$ satisfies the renewal equation
(\ref{renewal_M_tilde}), using its solution it may be expressed as
\begin{equation}\label{solMnt}
\tilde{M}_1(t)=\int^t_0 \tilde{b}_1(t-s)dm(s),
\end{equation}
and from Proposition \ref{prop_asymptotics}, recall that
\begin{equation}\label{chind}
\tilde{M}_1(t)  \longrightarrow \chi_1 =\frac{\int_0^\infty \tilde{b}_1(t) dt}{\nbE[\tau_1]},\quad t\to \infty ,
\end{equation}
where $\chi_1=\{\nbE[X]\nbE[L]\tilde{w}_{}(\delta)\}/\nbE[\tau_1]$ and $\tilde{w}_{}(\delta)=\int_0^\infty e^{-\delta x}\overline{W}(x)dx/\nbE[L]$ as given in \eqref{single_type_input_chi1}, Corollary \ref{Coro_single_type_of_input}. From  \cite{DR16}, we use the result of higher order expansions for the function $v(x)$ which is related to the renewal function as
\begin{equation}\label{vx}
v(x):=m(x)-\frac{x}{\nbE[\tau_1]}-\frac{\nbE[\tau_1^2]}{2\nbE[\tau_1]^2}.
\end{equation}
Since $F$ here is non-lattice (as it admits a density) which we suppose is light tailed, i.e. there exists $R>0$ such that {\bf (A1')} holds,
it admits the following expression
\begin{equation}\label{vxexpan}
v(x)=\sum^N_{j=1}\gamma_j e^{-z_j x} + o (e^{-z_N x}),
\end{equation}
where $z_j$'s are the solution of $\nbE[e^{z_j \tau_1}]=1$ which are in the range of $0 \leq \mathrm{Re}(z_j)\leq R$ for some $R>0$ and ordered as $\mathrm{Re}(z_j)\leq \mathrm{Re}(z_{j+1})$.  In order for (\ref{vxexpan}) to hold, we in addition require all roots $z_1,\ldots,z_N$ to be of multiplicity $1$, i.e. such that $\left.\frac{\partial}{\partial z} \nbE(e^{z \tau_1})\right|_{z=z_j}\neq 0$ (the condition is not necessary but it enables us to avoid some technical issues later), in which case one has
$$
\gamma_j=-\frac{1}{z_j \left.\frac{\partial}{\partial z}\nbE(e^{z \tau_1})\right|_{z=z_j}}, \quad j=1,\ldots,N,
$$
see \cite[Theorem 3]{DR16}. Although they are complex, the $z_j$'s actually come in pair as one sees that if $z_j$ verifies $\nbE[e^{z_j \tau_1}]=1$ then so does $\overline{z}_j$, so that the right-hand side of (\ref{vxexpan}) is in fact real. Furthermore, in the following result we need to write the term $o(e^{-z_N x})$ in (\ref{vxexpan}) in the form of
\begin{equation}\label{function_eta}
o(e^{-z_N x})=\eta(x) e^{-z_N x}
\end{equation}
for some function $\eta(x)$ such that $ \lim_{x\to \infty}\eta(x)=0$.

\begin{theorem}\label{theorem_expansion}
Let us assume that time delays $L_i$'s are ${\cal E}(\mu)$ distributed and that {\bf (A1')} holds. Then $\tilde{M}_1(t)$ in (\ref{M_tilde}) satisfies the following high order expansions
\begin{equation}\label{high_order_expansion}
\tilde{M}_1(t)=\chi_1 +A^\ast e^{-\mu t}+\sum^N_{k=1} B_{k}e^{-z_k t}+ o(e^{-z_N t}),
\end{equation}
where $A^\ast = A-\frac{\nbE[X]}{\nbE[\tau_1]}.\frac{1}{\mu+\delta} \L^{\tau}(-\mu)$ with
\begin{equation}\label{Ai}
A=-\nbE[X].\frac{\mu}{\mu+\delta}
\bigg[ \frac{\nbE[\tau_1^2]}{
2\nbE[\tau_1]^2}+\sum^N_{k=1}\gamma_k \frac{\mu}{z_k-\mu}+\mu\int^{\infty}_0 \eta(s)e^{(\mu-z_N)s} ds \bigg]\L^{\tau}
(-\mu),
\end{equation}
where $\eta(x)$ is defined by (\ref{function_eta}) and
\begin{equation}\label{Bki}
B_{k}=\nbE[X].\frac{\mu}{\mu+\delta}
\Big[\gamma_k \frac{z_k}{z_k-\mu} \Big]\L^{\tau}
(-z_k).
\end{equation}
\end{theorem}
\begin{proof}
See Section \ref{proof_theorem_expansion}.
\end{proof}

Note that in (\ref{high_order_expansion}) the $B_{k}$'s are explicit. On the other hand, $A$ in (\ref{Ai}) features an integral involving function $x\mapsto \eta(x)$ which is not explicit in general. This means that (\ref{high_order_expansion}) is explicit only if we truncate the expansion to the $i_0$th term where $i_0=\max\{j=1,\ldots,N|\ \mathrm{Re}(z_j)<\mu\}$. We may write the expansion in this way, however we prefer to keep a form as general as possible. Besides, we point out on a similar note that an expansion akin to (\ref{high_order_expansion}) was provided in \cite[Lemma 1]{BS75} for a general renewal reward process in the particular context where there is no time delay, under the weaker assumption that interarrival times and rewards admit a moment of order $1$.
\begin{remark}[Dependence of (\ref{high_order_expansion}) in $\delta$]\label{rem_dep_delta}
{\normalfont Upon inspecting (\ref{Ai}) and (\ref{Bki}) one notices that
$$
|A^*|, \ |B_{k}| \quad \le \frac{\kappa}{\mu+\delta},\quad k=1,\ldots,N,
$$
for all $\delta\ge 0$, where $\kappa>0$ is a constant independent from $\delta$. On further analysis, one also checks that when $\delta$ is {\it complex} and verifies $|\delta|<\mu$ then
\begin{equation}\label{dep_delta1}
|A^*|, \ |B_{k}| \quad \le \frac{\kappa}{\mu-|\delta|},\quad k=1,\ldots,N.
\end{equation}
In particular this inequality also holds when $\delta$ is 
{\it negative} and larger than $-\mu$. Hence, from (\ref{dep_delta1}), it is shown that $\tilde{M}_1(t)$ and $\chi_1$ are defined for such a complex $\delta$. This is particularly going to be the case in Section \ref{sec:workload}. Concerning the term $o(e^{-z_N t})$ in (\ref{high_order_expansion}), one carefully checks from the proof of Theorem \ref{theorem_expansion} that
\begin{equation}\label{dep_delta2}
|o(e^{-z_N t})|\le \frac{1}{\mu -|\delta|} \,\zeta(t) e^{- \mathrm{Re}(z_N)t},
\end{equation}
when $\delta\in\nbC$, $|\delta|<\mu$, for some function $\zeta(.)$ independent from $\delta$ verifying $\lim_{t\to\infty}\zeta(t)=0$.}
\end{remark}

\subsection{Asymptotics for the workload of the $G/M/\infty$ queue}\label{sec:workload}

An interesting application of the previous study of the one dimensional discounted compound delayed process $Z(t)$ is that we are able to find asymptotic results for the workload $D(t)$ of the infinite server queue when $k=1$. This $D(t)$ represents the time needed to empty the queue at time $t$ if there is no arrival afterwards. The distribution of this quantity was derived in \cite[Section 3]{BR69} for the $M/G/\infty$, but no results seem to have been obtained for a general arrival process with exponential service times, i.e. in the $G/M/\infty$. In particular, \cite{BR69} derived the distribution of the transient workload $D(t)$ in the case of Poisson arrivals with inhomogeneous intensity. The workload has the following expression
$$
D(t):= \sum_{i=1}^\infty (T_i+L_i-t)\nbu_{\{ T_i\le t <T_i+L_i\}},
$$
which is obtained from $\tilde{Z}(t,\delta):=e^{\delta t} Z(t)$ as:
\begin{equation}
D(t)=\left.-\frac{\partial}{\partial \delta}\tilde{Z}(t,\delta)\right|_{\delta=0}.
\label{workload_size}
\end{equation}
We assume in this section that all $X_{i,1}$ for $i\in\nbN$, are equal to one. In that case, $Z(t)$ in (\ref{Zjt}) is, when $\delta=0$, the size of this infinite server queue at time $t$.
\begin{figure}[!hbtp]%
\centering
\includegraphics[scale=0.8]{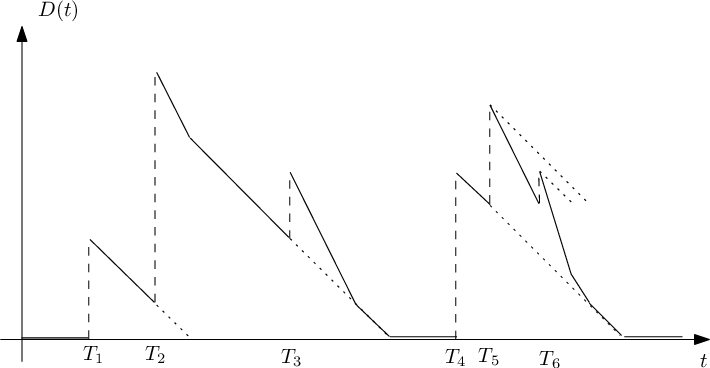}%
\caption{\label{fig:workload} Sample path of workload for the $G/G/\infty$ queue.}
\end{figure}
A sample path of $D(t)$ is depicted in Figure \ref{fig:workload}. Let us note that $D(t)$ is also the sum of the residual times for all services to be completed at time $t$. From an actuarial point of view, $D(t)$ may be interpreted as the remaining time before all current claims have been reported. In the following we shall obtain the limiting expectation of the workload and the covariance of the queue size and workload. We thus need to study the first two moments of $\tilde{Z}(t,\delta)$, i.e. quantities $\tilde{M}_{1}(t,\delta):=\tilde{M}_{n(1)}(t,\delta)=\EE[\tilde{Z}(t,\delta)]$ and $\tilde{M}_{2}(t,\delta):=\tilde{M}_{2n(1)}(t,\delta)=\EE[\tilde{Z}(t,\delta)^2]$, where we underline the dependence on $\delta$.

Here we assume that the service time $L$ is ${\cal E}(\mu)$ distributed, i.e.
\begin{equation}\label{assumption_light_tailed}
\EE[e^{x L}]=\frac{\mu}{\mu-x},\quad \forall x\in(-\infty,\mu ),
\end{equation}
so that this is the $G/M/\infty$ queue, and that interarrival times are light tailed, i.e. {\bf (A1')} holds for some $R>0$. To begin, two lemmas are first required. We need to define for $r>0$, the disc $D_r$ centered at $0$ with the radius $r$, included in $\nbC$, by
$$
D_r:=\{ z\in \nbC|\ |z|\le r\}.
$$
\begin{lemm}\label{Lemma_analytic}
Let $a<\mu$ and let us suppose that  {\bf (A1)} and {\bf (A1')} hold. For all $t>0$, $\tilde{M}_{1}(t,\delta)$ and $\tilde{M}_{2}(t,\delta)$ are respectively defined on $D_a$ and $D_{a/2}$. Furthermore, $\delta\mapsto \tilde{M}_{1}(t,\delta)$ and $\delta\mapsto \tilde{M}_{2}(t,\delta)$ are analytic on those sets, hence a fortiori at $\delta=0$.
\end{lemm}
Note that one implication of the above lemma is that $\tilde{M}_{1}(t,\delta)$ and $\tilde{M}_{2}(t,\delta)$ (and, hence $\tilde{Z}(t,\delta)$) are defined for some complex values of $\delta$, and in particular for negative values (not only for $\delta \ge 0$). This is especially handy to express the workload as (\ref{workload_size}) and to be able to define analyticity of $\tilde{M}_{1}(t,\delta)$ and $\tilde{M}_{2}(t,\delta)$ at $\delta=0$, which is needed to differentiate with respect to $\delta$ at $0$.
\begin{proof}
See Section \ref{proof_workload}.
\end{proof}
\begin{lemm}\label{lemma_uniform_convergence}
Let us suppose that {\bf (A1)} and {\bf (A1')} hold and let $a<\mu$. $\delta\mapsto\tilde{M}_{1}(t,\delta)$ and $\delta\mapsto\tilde{M}_{2}(t,\delta)$ uniformly converge to $\delta\mapsto\chi_1(\delta)$ and $\delta\mapsto\chi_2(\delta)$ respectively on $D_a$ and $D_{a/2}$ as $t\to+\infty$.
\end{lemm}
\begin{proof}
See Section \ref{proof_workload}.
\end{proof}

Now we are ready to provide some results for the long-term behaviour of the expected workload, and the covariance function of the workload and the queue size in the following.
\begin{theorem}\label{workload}
Let us suppose that {\bf (A1)} and {\bf (A1')} hold. In the $G/M/\infty$ queue, the limiting expected workload is given by
\begin{equation}\label{limiting_expected_workload}
\lim_{t\to\infty}\nbE[D(t)]=\frac{1}{\mu^2 \nbE[\tau_1]}=\frac{\nbE[L^2]}{2\nbE[\tau_1]},
\end{equation}
and the limiting covariance of workload and queue size is given by
\begin{equation}\label{limiting_cov_workload}
\lim_{t\to\infty}\Cov[D(t), Z_1(t,0)]=\frac{1}{\mu^2 \nbE[\tau_1]}\left[1+ \frac{\L^\tau(\mu)}{1-\L^\tau(\mu)}-\frac{1}{\mu \nbE[\tau_1]}\right].
\end{equation}
\end{theorem}
\begin{proof} See Section \ref{proof_workload}.
\end{proof}
\begin{remark}
{\normalfont When $k=1$, utilizing \eqref{workload_size}, it is possible to get an expression of the expected workload and covariance of workload and queue size at time $t$ in the $M/G/\infty$ queue as well. This is done thanks to the (easily verified) relations
\begin{eqnarray*}
\nbE[D(t)]&=&\left.-\frac{1}{s}\frac{\partial}{\partial \delta}\tilde{\psi}(s,t)\right|_{\delta=0,s=0},\\
\Cov[D(t), Z_1(t,0)]&=&\left. \left[ \frac{\partial}{\partial s}\left[ -\frac{1}{s} \frac{\partial}{\partial \delta}\tilde{\psi}(s,t)  \right] - \left[  -\frac{1}{s}\frac{\partial}{\partial \delta}\tilde{\psi}(s,t) \right]. \left[ \frac{\partial}{\partial s}\tilde{\psi}(s,t)\right] \right]\right|_{\delta=0, s=0},
\end{eqnarray*}
where $\tilde{\psi}(s,t)$ is given by \eqref{transient_Poisson} with $M^*_{t,X}(s)=\overline{W}(t)+ \int_t^\infty e^{s e^{-\delta v}}dW(s)$. In contrast to Theorem \ref{workload}, justification of the above formulas is much easier as one does not have to justify interchange of expectation and derivation with respect to $\delta$, which is the core step in the proof of Theorem \ref{workload}, and is done with the help of Lemmas \ref{Lemma_analytic} and \ref{lemma_uniform_convergence}.}
\end{remark}

\section{Applications}\label{sec:App}
\subsection{Queues with different service times within a batch}\label{sec:different_service_times}
The queueing model introduced in Section \ref{intro} features a queue where customers arrive in a batch of size $X_{i,j}$ with class $j$ at time $T_i$. Each customer in this batch has the same service time $L_{i,j}$ within the same class $j$ for $j=1,\ldots,k$. One may argue that this scenario is not much realistic since each customer $\l\in\{1,\ldots,X_{i,j} \}$ may have different service times $L_{ij\l}$. Here $(L_{ij\l})_{(i,j,\l) \in \nbN^3}$ are independent random variables, with $(L_{ij\l})_{\l\in  \nbN}$ 
identically distributed for all $i\in\nbN$ and $j=1,\ldots,k$, so that customers within a batch get different service times. 

It can be shown that this (more realistic) situation is essentially expressed in the form of our model by constructing a "larger" vector $X=(X_1,\ldots,X_k)$ as follows. For illustrative purposes, recall the situation depicted in Figure \ref{example_queue} where $M$ customers arrive in a batch, but now let us consider that customer $\ell \in \{1,\ldots ,X_{i,j} \}$ has a service time $L_{ij\ell}$ instead of $L_{i,j}$.
In other words, let such a sequence $(L_{ij\l})_{(i,j,\l) \in \nbN^3}$ be given and let $p_{j\l}$ be the probability of a customer being in class $j\in\{1,\ldots,k\}$ with a service time $L_{j\l}$, for some generic random matrix $(L_{j \l})_{j=1,\ldots,k,\ \l=1,\ldots,M}$. This situation is then modelled thanks to the one described in Section \ref{intro} by considering a vector $X=(X_{j,\l})_{j=1,\ldots,k,\ \l=1,\ldots,M}$ of length $kM$ (written as a matrix) such that
\begin{multline*}
X=(X_{j,\l})_{j=1,\ldots,k,\ \l=1,\ldots,M}\\
\sim {\cal D}\left( (Y_{j,\l})_{j=1,\ldots,k,\ \l=1,\ldots,M}|\ Y_{j,\l}\in \{0,1\},\ \forall (j,\l)\in \{1,\ldots,k \}\times \{1,\ldots,M \}\right),
\end{multline*}
where $(Y_{j,\l})_{j=1,\ldots,k,\ \l=1,\ldots,M}$ be a matrix with distribution ${\cal M}(M, (p_{j\l})_{j=1,\ldots,k,\ \l=1,\ldots,M})$, i.e. a random vector of length $kM$ with a multinomial distribution with parameter $M$ and a probability vector $(p_{11},\ldots,p_{1M},p_{21},\ldots,p_{2M},\ldots,p_{k1},\ldots,p_{kM})$.


\subsection{Infinite server queues in tandem}
To further illustrate the versatility of the present model, let us now consider the following two infinite server queues in tandem setup. We suppose that each batch arriving at time $T_i$ contains $X_{i,j}$ customers where there are $k$ classes of customers. Once customer of class $j\in\{1,\ldots,k\}$ arrives in the first queue, he is served during a {\it deterministic} time $L^1_{i,j}$. Upon completion of the service, i.e. after leaving the first queue, he is then directly sent to the second queue (again with an infinite number of servers) where he is served during a time $L^2_{i,j}$. This kind of successive treatments of queues is easily observed in the claims payment 
process in actuarial science. In general, there are time delays between the time of incurral of the claim and the time of receipt of payment. Of course, for the insurers, they are concerned with the time from receipt of notification of the claim until approval or payment. It is natural that each stage for one claim has different processing times (i.e. different distribution for time delays). In actuarial practice, some stages in the claim settlement process are completed on a scheduled time in compliance with accounting/regulation rules. Hence, a deterministic delay $L^{1}_{i,j}$ (for each class $j$) is an appropriate setting in such case. See \cite{W90} for detailed discussion related to the insight of queueing theoretic tools into the claims payment process. 

Again, let us consider the case where $M$ customers arrive in a batch at time $T_i$ (a finite size of batch is assumed only for illustrative purposes). Within the same type of class $j$
of which size $X_{i,j}$, all have the same service times $L^1_{i,j}$ and $L^2_{i,j}$. Certainly, as explained in Section \ref{sec:different_service_times}, different service times within a batch may also be available. We assume that $(L^2_{i,j})_{i\in \nbN,\ j=1,\ldots,k}$ are  independent, and that, as usual, $L^2_{i,1},\ldots, L^2_{i,k}$ have different distribution for each class; In the same vein, service times $L^1_{i,1},\ldots, L^1_{i,k}$ in the first queue are all deterministic but are different for each class. This is represented in Figure \ref{example_queue_tandem}.
\begin{figure}[!hbtp]%
\centering
\includegraphics[scale=0.8]{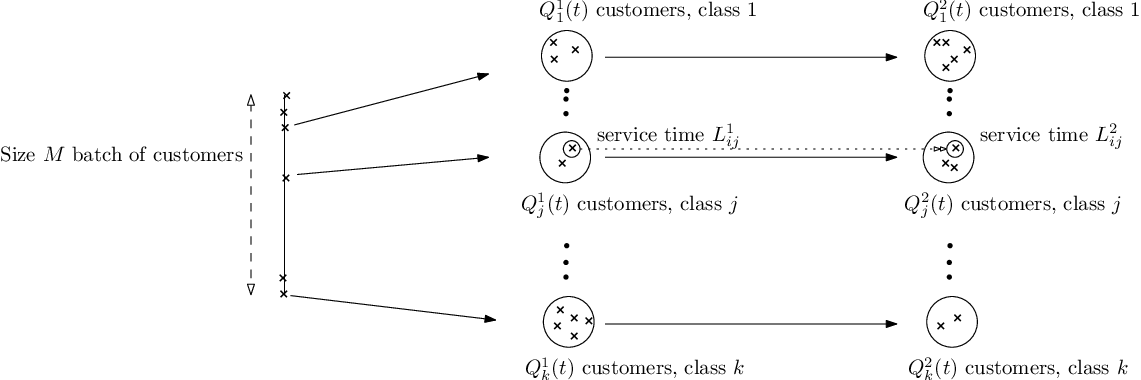}%
\caption{\label{example_queue_tandem} Two queues in tandem.}
\end{figure}
We are interested in the number of customers of class $j$ in the second queue at time $t$ which is denoted by $Q^2_j(t)$. It is not hard to see that one has the expression
 \begin{equation}\label{Q2_tandem}
 Q^2_j(t)=\sum_{i=1}^\infty X_{i,j}\nbu_{\{T_i+L^1_{i,j}\le t < T_i+L^1_{i,j}+L^2_{i,j}\}},\quad j=1,\ldots,k,
 \end{equation}
 where $X_i=(X_{i,1},\ldots,X_{i,k})\sim {\cal M}(M,p_1,\ldots,p_k)$. Let us then introduce, for $i\in\nbN$, the $\nbN^{2k}$ sized vector
$$
X_i'=(X_{i,1}',\ldots,X_{i,2k}'):=(X_{i,1},\ldots,X_{i,k},X_{i,1},\ldots,X_{i,k})
$$
(i.e. the vector $X$ is concatenated with itself) as well as the $[0,+\infty)^{2k}$ sized vector
\begin{equation}\label{Lprimeij}
L_{i,j}'=\left\{
\begin{array}{c l}
L^1_{i,j},&\quad j=1,\ldots,k,\\
L^1_{i,j-k}+L^2_{i,j-k},&\quad j=k+1,\ldots,2k.
\end{array}
\right.
\end{equation}
One important remark is that, since the $L^1_{i,j}$'s are deterministic, the sequence $(L_{i,j}')_{i\in \nbN,\ j=1,\ldots,2k}$ has {\it independent} components. Hence, this model can be expressed under the setting of our model as described in Section \ref{intro} but with $X'$ and $(L_{i,j}')_{i\in \nbN,\ j=1,\ldots,2k}$ in lieu of $X$ and $(L_{i,j})_{i\in \nbN,\ j=1,\ldots,k}$. To be specific, let us define the $\nbN^{2k}$ valued process $Z(t)=(Z_1(t),\ldots,Z_{2k}(t))$ with
$$
Z_j(t)=\sum_{i=1}^\infty X_{i,j}'\nbu_{\{T_i\le t < T_i+L_{i,j}'\}},\quad j=1,\ldots,2k,
$$
where $L_{i,j}'$ are defined in (\ref{Lprimeij}). Then, one has in particular $Z_j(t)=Q^1_j(t)$ for $j=1,\ldots,k$. One also notes from \eqref{Q2_tandem} that $Q^2_j(t)=Z_{j+k}(t)-Z_j(t)$ for $j=1,\ldots,k$. The mgf of $Q^2(t)=(Q^2_1(t),\ldots,Q^2_k(t))$ can then be expressed in terms of the mgf of $Z(t)$ by
\begin{equation}\label{mgf_Q2}
\nbE\left[ e^{<s,Q^2(t)>}\right]=\nbE\left[ e^{<(-s,s),Z(t)>}\right],\quad s=(s_1,\ldots,s_k)\in\nbR^k ,
\end{equation}
where $(-s,s):=(-s_1,\ldots,-s_k,s_1,\ldots,s_k)\in \nbR^{2k}$. The consequence of \eqref{mgf_Q2} is that
\begin{itemize}
\item If arrivals occur according to a Poisson process with intensity $\lambda$, then the mgf of $Q^2(t)$ is {\bf explicit} thanks to Proposition \ref{prop_case_Poisson} that
$$\nbE\Big[ e^{<s,Q^2(t)>}\Big]= \exp\left[ \lambda \int_0^t \left( M^*_{v,X'}((-s,s))-1\right) dv\right].$$
\item If arrival processes are general but satisfy {\bf (A1)} then less information is available on the transient distribution, however one has from Theorem \ref{theo_conv_distrib} that $Q^2(t)$ converges in distribution to some light tailed random vector as $t\to\infty$, and that some simple bounds on the joint moments of this limiting random vector are available from Proposition \ref{tildeMnRn}. 
\end{itemize}


\section{Proofs}\label{sec:Proofs}
{{\bf Proof of Lemma \ref{lemma_density_upper_bound}}.
When $\tau_1$ admits a pdf $f(\cdot)$ then density $ t\mapsto u(t)$ of renewal function $t\mapsto m(t)$ satisfies a renewal equation
\begin{equation}
u(x)=f(x)+\int_0^x u(y) f(x-y)dy,\quad x\ge 0,
\label{renewal_function_density}
\end{equation}
(e.g. see Equation (3.6) of \cite{Feller}).
Since {\bf (A1)} holds, by \cite[Lemma, p.359]{Feller} (\ref{renewal_function_density}) admits a unique solution bounded on finite intervals given by (\ref{Expression_density_renewal}). Also, the derivative $m'(t)=u(t)$ verifies $\lim_{t\to\infty}m'(t)=1/\nbE[\tau_1]$, see \cite[Theorem 2, p.367]{Feller}, and is thus  bounded above by some constant $C$. \hfill $\Box$}

\subsection{Proof of Proposition \ref{prop_asymptotics}}\label{proof_prop_asymptotics}
Since $\tilde{M}_n(t)$ satisfies the renewal equation in (\ref{renewal_M_tilde}), asymptotic result in (\ref{asymp_tilde_M}) is a direct consequence of Smith's renewal theorem (see \cite{S58}, \cite{BR69a} for example), provided that we prove that $\int_0^\infty \tilde{b}_n(y)dy$ or equivalently $\int_0^\infty \ph_{\l,n}(y)dy$ is finite for all $n\in\nbN^k$ and $\l<n$. We shall demonstrate this by induction on $n\in\nbN^k$.
First, consider the case of $n=n(i)$ for some $i\in\{1,\ldots,k\}$. From (\ref{bnt}), we first calculate $\int_0^\infty \tilde{b}_n(y)dy$. 
But, we get from (\ref{bomega}) that $\int^\infty_0 e^{\delta z} \overline{\omega}_{\delta,i}(z)dz$ $=\int^\infty_0 e^{\delta z} \int^\infty_z e^{-\delta y}dW_i(y)dz$ $=\delta^{-1}\{1-\nbE[e^{-\delta L_i}]\}$. Then, the following integration yields
\begin{align}
\int_0^\infty \tilde{b}_n(y)dy &=\nbE[X_i].\,\int^\infty_0 e^{\delta y} \int^y_0 e^{-\delta x} \overline{\omega}_{\delta,i}(y-x)dF(x)dy \nonumber\\
&=\nbE[X_i].\, \int^\infty_0 e^{-\delta x} \int^\infty_x e^{\delta y} \overline{\omega}_{\delta,i}(y-x) dy dF(x)\nonumber\\
&=\nbE[X_i].\, \delta^{-1}
\left\{1-\nbE[ e^{-\delta L_i}]\right\}
<\infty,
\label{int_b_tilde}
\end{align}
or equivalently
\begin{equation*}\label{int_b_tildeA}
\int_0^\infty \tilde{b}_n(y)dy = \nbE[X_i]\nbE[L_i]\int^\infty_0 e^{-\delta x}
\frac{\overline{W}_i(x)}{\nbE[L_i]}dx=
\nbE[X_i]\nbE[L_i]\tilde{w}_{1,i}(\delta),
\end{equation*}
where $w_{1,i}(x)$ is an equilibrium pdf of $L_i$ defined as $w_{1,i}(x)=\overline{W}_i(x)/\nbE[L_i]$ and its Laplace transform is $\tilde{w}_{1,i}(s)=\int^\infty_0 e^{-sx}w_{1,i}(x)dx$.

Moreover,  recall (\ref{Eq36ni}), and by Smith's theorem, it satisfies
\[
M_n(t)
\sim
\frac{\nbE[X_i]}{\nbE[\tau_1]}
\bigg[ \int_0^\infty
e^{\delta y}\overline{\omega}_{\delta, i}(y)dy\bigg]
e^{-\delta t},\qquad t\to\infty.
\]
 In other words, one identifies
\begin{equation*}\label{prop_asymptotics_expression_C_n_i}
\chi_n=\chi_{n(i)}=\frac{\nbE[X_i]}{\nbE[\tau_1]}\bigg[ \int_0^\infty e^{\delta y}\overline{\omega}_{\delta, i}(y)dy\bigg].
\end{equation*}

We now assume for all $\l<n$ that $\tilde{M}_\ell(t)\rightarrow \chi_\ell<+\infty$ as $t\to\infty$ with $\chi_\ell$ defined as in (\ref{asymp_tilde_M}). Hence $t\mapsto \tilde{M}_\ell(t)$ is bounded for all $\ell<n$ by some constant $K_\ell=\sup_{t\geq 0} \tilde{M}_\ell(t)$. Hence simple algebraic computation results in the upper bound for (\ref{phi_l}) as
\begin{eqnarray*}
\ph_{\l ,n}(t)&\le & K_\ell \nbE\bigg[ e^{(\eta_n-\eta_\ell) \delta (t-\tau_1)} \prod_{j\in C_{\ell ,n}} \overline{\omega}_{(n_j-\ell_j)\delta,j}(t-\tau_1).
\nbu_{[\tau_1<t]}\bigg]\\
&=& K_\ell \nbE\bigg[ e^{(\eta_n-\eta_\ell) \delta (t-\tau_1)} \prod_{j\in C_{\ell ,n}} \bigg[ \int_{t-\tau_1}^\infty e^{-(n_j-\ell_j)\delta y} dW_j(y)\bigg].\nbu_{[\tau_1<t]}\bigg]\\
&\le & K_\ell \nbE\bigg[ e^{(\eta_n-\eta_\ell) \delta (t-\tau_1)} \prod_{j\in C_{\ell ,n}} \bigg[ e^{-(n_j-\ell_j)\delta (t-\tau_1)} \overline{W}_j(t-\tau_1) \bigg].\nbu_{[\tau_1<t]}\bigg]\\
&=& K_\ell \nbE\bigg[  \prod_{j\in C_{\ell ,n}}   \overline{W}_j(t-\tau_1) .\nbu_{[\tau_1<t]}\bigg].
\end{eqnarray*}
Then integrating $\ph_{\l,n}(t)$ from $0$ and $\infty$ yields
\begin{equation*}\label{bet}
\int_0^\infty \ph_{\l ,n}(t) dt
\le  K_\ell \nbE\bigg[ \int_0^\infty \prod_{j\in C_{\ell ,n}}   \bigg[\overline{W}_j(t-\tau_1)\bigg].\nbu_{[\tau_1<t]}\, dt\bigg]=
 K_\ell \int_0^\infty \prod_{j\in C_{\ell ,n}}  \overline{W}_j(t)  \, dt,
\end{equation*}
and by Holder's inequality, one finds
\begin{eqnarray}
\int_0^\infty \ph_{\l ,n}(t) dt
&\le & K_\ell \prod_{j\in C_{\ell ,n}} \bigg[\int_0^\infty \overline{W}_j(t)^{|C_{\ell ,n}|}\, dt\bigg]^{1/|C_{\ell ,n}|}\nonumber\\
&\le & K_\ell \prod_{j\in C_{\ell ,n}} \left[\int_0^\infty \overline{W}_j(t) \, dt\right]^{1/|C_{\ell ,n}|}=K_\ell \prod_{j\in C_{\ell ,n}} \nbE[L_j]^{1/|C_{\ell ,n}|}\nonumber\\
&\le & K_\ell\, \max_{j\in C_{\ell ,n}} \nbE[L_j] <\infty , \label{bound_phi_l}
\end{eqnarray}
where  $|C_{\ell ,n}|$ denotes the cardinal of set $C_{\ell ,n}$. Hence from (\ref{exp_tilde_b}) we deduce that $\int_0^\infty \tilde{b}_n(y)dy$ is also finite, and the induction is complete.

\subsection{Proof of Proposition \ref{tildeMnRn}}\label{proof_tildeMnRn}

Since $m(t)$ admits $u(t)$ as a density, one has from (\ref{renewal_M_tilde}) that $\tilde{M}_n(t)=\int_0^t \tilde{b}_n(y)u(t-y)dy$, and in turn, from Lemma \ref{lemma_density_upper_bound} we arrive at the following upper bound
\begin{equation}
\tilde{M}_n(t)\le C\int_0^\infty \tilde{b}_n(y)dy.
\label{upper_bound_M_tilde}
\end{equation}
Combining (\ref{exp_tilde_b}) and (\ref{bound_phi_l}) yields the following upper bound
\[
\int_0^\infty \tilde{b}_n(y)dy\ \le \ \sum_{\ell < n} {{n_1}\choose{\ell_1}}\cdots
{{n_k}\choose{\ell_k}}\nbE\bigg[ \prod_{j=1}^k X_j^{n_j-\ell_j}\bigg] K_\ell\, \max_{j\in C_{\ell ,n}} \nbE[L_j],
\]
where we recall that $K_\ell=
\sup_{t\ge 0} \tilde{M}_\ell(t)$
(see the proof of Proposition \ref{prop_asymptotics}).
Thus the above inequality together with (\ref{asymp_tilde_M}) and
(\ref{upper_bound_M_tilde}) yields
(\ref{bound_Cn}) and (\ref{bound_M_tilde}) respectively with $(R_n)_{n\in \nbN^k}$ defined in
(\ref{rec_Rn}), provided
we initialize value of $R_n$ when $n=n(i)$ for $i\in\{1,\ldots,k\}$.
This is done by again using upper bound
(\ref{upper_bound_M_tilde}) and
remembering that $\int_0^\infty \tilde{b}_n(y)dy$
is obtained by (\ref{int_b_tilde})
when $n=n(i)$.


\subsection{Proof of Theorem \ref{theo_conv_distrib}}\label{proof_theo_conv_distrib}
Let $P(x_1,\ldots,x_k)=\sum_{\eta_n\le K} a_n x_1^{n_1}\cdots x_k^{n_k}$ be a nonnegative polynomial in the variables $x_1$\ldots$x_k$ of degree $K$.
One has then that $\sum_{\eta_n\le K} a_n \nbE\left[\prod_{i=1}^k \tilde{Z}_i^{n_i}(t)\right]=\nbE\left[P(\tilde{Z}_1(t),\ldots,\tilde{Z}_k(t))\right]\ge 0$ for all $t$, which, from Proposition \ref{prop_asymptotics}, yields $\sum_{\eta_n\le K} a_n\chi_n \ge 0$ as $t\to\infty$. By the Riesz-Haviland theorem (see \cite{Haviland}), we deduce that sequence $(\chi_n)_{n\in \nbN^k}$ is a sequence of moments associated to some random variables $\mathcal{Z}_\infty=(\mathcal{Z}_{\infty, 1},\ldots,\mathcal{Z}_{\infty,k})$. In \cite{Haviland}, the proofs are given for two dimensional random variables $(X,Y)$ for convenience but the result holds for any $n$-dimensional random variables.

Next we shall show that the mgf of $\tilde{Z}(t)$ exists and converges to the mgf of $\mathcal{Z}_\infty$ as $t \to \infty$. To this end, we note that the mgfs of $\tilde{Z}(t)$ and of $\mathcal{Z}_\infty$ respectively defined by \eqref{mgf_Z_tilde} and denoted by $\psi_\infty(s)$ verify by Fubini's theorem
\begin{eqnarray}
\tilde{\psi}(s,t)&=& \nbE[e^{<s,\tilde{Z}(t)>}] =\nbE \bigg[\prod^k_{j=1} e^{s_j \tilde{Z}_j(t)} \bigg]=\nbE \left[\prod^k_{j=1} \left(\sum_{n_j=0}^\infty\frac{\left[s_j \tilde{Z}_j(t)\right]^{n_j}}{n_j!}\right) \right]\nonumber\\
&=&
\nbE\bigg[\sum_{n\in \nbN^k} \prod_{i=1}^k \frac{[s_i\tilde{Z}_i(t)]^{n_i}}{n_i!}\bigg]=\sum_{n\in \nbN^k} \prod_{i=1}^k \frac{s_i^{n_i}}{n_i!}\nbE\bigg[\prod^k_{j=1} \tilde{Z}_j(t)^{n_j}\bigg]=
\sum_{n\in \nbN^k} \prod_{i=1}^k \frac{s_i^{n_i}}{n_i!} \,\tilde{M}_n(t),\label{mgfedtZt}\\
\psi_\infty(s)&=& \nbE\left[ e^{<s,\mathcal{Z}_\infty>}\right]=\sum_{n\in \nbN^k} \prod_{i=1}^k \frac{s_i^{n_i}}{n_i!} \,\chi_n,  \label{mgfLinf}
\end{eqnarray}
for $t\ge 0$ and $s=(s_1,\ldots,s_k)\in \nbR^k$ in the neighborhood of $(0,\ldots,0)$. Let us prove this convergence of mgf's in the two separate cases of {\bf (A2)} when the $X_i$'s are upper bounded by some deterministic $M$, or are NBU.\\
\noindent{\bf Case (i): $X_i$'s are upper bounded.} Let us first suppose that $0\le X_i \le M$ a.s. for all $i=1,...,k$, for some deterministic constant $M$. To show that $\lim_{t\to \infty}\tilde{\psi}(s,t)=\psi_\infty(s)$ by the dominated convergence theorem, it suffices to prove that $\tilde{M}_n(t)$ is bounded such as
\begin{equation}\label{upper_bound_conv_distrib}
\tilde{M}_n(t)\le C U_n:=C(Mm_L e^k)^{\eta_n} \prod_{i=1}^k n_i!,\quad \forall n\in\nbN^k,\quad \forall t\ge 0,
\end{equation}
where $m_L:=\max\left(1, \max_{i=1,\ldots,k}\nbE [L_i]\right)$. Since
\[
\sum_{n\in \nbN^k} \prod_{i=1}^k \frac{|s_i|^{n_i}}{n_i!}\, U_n= \prod_{i=1}^k \left( \sum_{n_i=0}^\infty |s_iMm_L e^k|^{n_i}\right)
\] converges for
\[
s=(s_1,\ldots,s_k)\in J:=\left[-\frac{1}{Mm_L e^k}, \frac{1}{Mm_L e^k}\right]^k,
\]
the dominated convergence theorem yields $\tilde{\psi}(s,t)\longrightarrow \psi_\infty(s)$ when $t\to\infty$ for $s\in J$.

Hence, we shall prove (\ref{upper_bound_conv_distrib}) by induction. Without loss of generality, we may assume that upper bounding constant $M$ verifies $M\ge 1$, otherwise one may replace $M$ by $\max(1,M)$ in what follows. Recall that in Proposition \ref{tildeMnRn}, we have already proved $\tilde{M}_n(t)\leq C R_n$ where $R_n$ is defined in (\ref{rec_Rn}). Thus we shall essentially show that $R_n\le U_n$ for all $n\in \nbN^k$, so that (\ref{upper_bound_conv_distrib}) holds. We start by $n=n(i)$ for $i\in \{1,\ldots,k\}$. Since $e$ is larger than $1$, (\ref{rec_Rn}) is bounded as
$$
R_{n(i)}= \nbE[X_i]
\delta^{-1}
\left\{1-\nbE\left[ e^{-\delta L_i}\right]\right\}\le M \nbE[L_i]\le M m_L e^k=U_{n(i)},
$$
where the first inequality is due to $
\delta^{-1}
\left\{1-\nbE\left[ e^{-\delta L_i}\right]\right\}=\int^\infty_0 e^{-\delta x}\overline{W}_i(x)dx\leq \int^\infty_0 \overline{W}_i(x)dx$.
Let us now suppose that $n$ is such that $R_{\ell}\le U_{\ell}$ for all ${\ell}<n$. Using (\ref{rec_Rn}) as well as the induction assumption we get 
\begin{equation}\label{upper_bound_conv_distrib2}
R_n\le  \sum_{\ell < n} {{n_1}\choose{\ell_1}}\cdots{{n_k}\choose{\ell_k}} \nbE\bigg[ \prod_{j=1}^k X_j^{n_j-\ell_j}\bigg] \max_{i\in C_{\ell ,n}} \nbE[L_i].\, U_{\ell}\le m_L \sum_{\ell < n} {{n_1}\choose{\ell_1}}\cdots{{n_k}\choose{\ell_k}} M^{\eta_n - \eta_{\ell}}
.\, U_{\l}.
\end{equation}
But, $\l<n$ implies $\eta_n - \eta_{\ell}\ge 1$
and $m_L$ and $e$ are larger than $1$, the following inequality is valid
$$
m_L M^{\eta_n - \eta_{\ell}}
\le (m_L  M)^{\eta_n - \eta_{\ell}}
 (e^k)^{\eta_n - \eta_{\ell}-1}=(m_LM e^k)^{\eta_n - \eta_{\ell}} e^{-k}.
$$
Substituting the above inequality and $U_\ell=(Mm_L e^k)^{\eta_\ell} \prod_{i=1}^k \ell_i!$ into (\ref{upper_bound_conv_distrib2}), the right-hand side of (\ref{upper_bound_conv_distrib2}) is now bounded by
\begin{eqnarray}
R_n & \le & \sum_{\ell < n} {{n_1}\choose{\ell_1}}\cdots{{n_k}\choose{\ell_k}} (m_LM e^k)^{\eta_n - \eta_{\ell}} e^{-k} (Mm_L e^k)^{\eta_\l} \prod_{i=1}^k \l_i!\nonumber\\
&= & (Mm_L e^k)^{\eta_n}\bigg[ \sum_{\ell < n} \prod_{i=1}^k \frac{n_i!}{(n_i-\l_i)!}\bigg] e^{-k}=(Mm_L e^k)^{\eta_n}\bigg[\prod_{i=1}^k n_i!\bigg]\bigg[ \sum_{\ell < n} \prod_{i=1}^k \frac{1}{(n_i-\l_i)!}\bigg] e^{-k}\nonumber\\
&=& U_n \bigg[ \sum_{\ell < n} \prod_{i=1}^k \frac{1}{(n_i-\l_i)!}\bigg] e^{-k}. \label{upper_bound_conv_distrib3}
\end{eqnarray}
We then conclude by noticing that
\begin{eqnarray*}
\sum_{\ell < n} \prod_{i=1}^k \frac{1}{(n_i-\l_i)!}&\le &\sum_{\ell_i\le n_i,\, i\in\{1,\ldots,k\}} \prod_{i=1}^k \frac{1}{(n_i-\l_i)!} \\
&= &\prod_{i=1}^k \bigg[ \sum_{\l_i=1}^{n_i}\frac{1}{(n_i-\l_i)!}\bigg]= \prod_{i=1}^k \bigg[ \sum_{\l_i=1}^{n_i}\frac{1}{\l_i!}\bigg]\le \prod_{i=1}^k \bigg[ \sum_{\l_i=1}^{\infty}\frac{1}{\l_i!}\bigg]=e^k,
\end{eqnarray*}
which, plugged into (\ref{upper_bound_conv_distrib3}), yields $R_n\le U_n$. Therefore, by the dominated
convergence theorem, $\tilde{\psi}(s,t)$ in (\ref{mgfedtZt}) converges to $\psi_\infty(s)$ in (\ref{mgfLinf}) as $t\rightarrow \infty$.\\
\noindent{\bf Case (ii): $X_i$'s are NBU.} We are aiming here to obtain a uniform bound similar to \eqref{upper_bound_conv_distrib}. Let us define the rv $M:=\sum_{i=1}^k X_i$. Since $X_i$, $i=1,...,k$, are all NBU, \cite[Proposition C.11, p.165]{MO07} yields that $M$ is also NBU. Furthermore, \cite[Proposition A.6, p.197]{MO07} entails that one can have some control on the higher order moments of $M$ by its first moment thanks to the following inequality:
\begin{equation}\label{moment_bound_NBU}
\nbE(M^m)\le m! [\nbE(M)]^m,\quad \forall m\in \nbN^* .
\end{equation}
The above inequality is the starting point to find the upper bound for $\tilde{M}_n(t)$. Let us prove that
\begin{multline*}
\prod_{i=1}^k |s_i|^{n_i}R_n\le U_n:= \eta_n!\frac{1}{(2k)^{\eta_n}},\quad \forall n\in \nbN^k,\\
\forall s=(s_1,...,s_k)\in J':=\left[ -\frac{1}{2k\nbE(M) m_L e^k};\frac{1}{2k\nbE(M) m_L e^k}\right]^k ,
\end{multline*}
where $m_L:=\max\left(1, \max_{i=1,\ldots,k}\nbE [L_i]\right)$ is defined in the previous case. As in the previous case, we proceed by induction. Starting by $n=n(j)$, $j\in \{1,...,k\}$, we have that (again since $m_L$ and $e$ are larger than $1$)
\begin{multline*}
\prod_{i=1}^k |s_i|^{n_i} R_{n(j)}= |s_j|\nbE[X_i]
\delta^{-1}
\left\{1-\nbE\left[ e^{-\delta L_i}\right]\right\}\le |s_j| \nbE (M) \nbE[L_i]\\
\le |s_j|\nbE(M) m_L e^k\le \frac{1}{2k}=U_{n(i)},\quad \forall s=(s_1,...,s_k)\in J' ,
\end{multline*}
as one has indeed that $\eta_{n(i)}=1$. Supposing now that $\prod_{i=1}^k |s_i|^{\ell_i}R_\ell\le U_\ell$ for all $\ell<n$ and $s\in J'$, it is noted that we have, thanks to \eqref{moment_bound_NBU}, the following inequality
$$
\nbE\bigg[ \prod_{j=1}^k X_j^{m_j}\bigg] \le \nbE[M^{\eta_m}]\le \eta_m! [\nbE(M)]^{\eta_m},\quad \forall m=(m_1,...,m_k)\in \nbN^k ,
$$
which, similarly to \eqref{upper_bound_conv_distrib2}, yields that (using again that $m_L$ and $e$ are larger than $1$ and $ \eta_n - \eta_{\ell}\ge 1$ when $\ell<n$)
\begin{eqnarray}
\prod_{i=1}^k |s_i|^{n_i} R_n&=&  \sum_{\ell < n} {{n_1}\choose{\ell_1}}\cdots{{n_k}\choose{\ell_k}} \prod_{i=1}^k |s_i|^{n_i-\ell_i} \nbE\bigg[ \prod_{j=1}^k X_j^{n_j-\ell_j}\bigg] \max_{i\in C_{\ell ,n}} \nbE[L_i].\, \prod_{i=1}^k |s_i|^{\ell_i} R_{\ell}\nonumber\\
&\le &  \sum_{\ell < n} {{n_1}\choose{\ell_1}}\cdots{{n_k}\choose{\ell_k}} \prod_{i=1}^k |s_i|^{n_i-\ell_i} \nbE\bigg[ \prod_{j=1}^k X_j^{n_j-\ell_j}\bigg] \max_{i\in C_{\ell ,n}} \nbE[L_i].\, U_{\ell}\nonumber\\
&\le &  \sum_{\ell < n} {{n_1}\choose{\ell_1}}\cdots{{n_k}\choose{\ell_k}} \prod_{i=1}^k |s_i|^{n_i-\ell_i} (\eta_n - \eta_{\ell})![\nbE(M)]^{\eta_n - \eta_{\ell}} m_L
.\, U_{\l} \nonumber\\
&\le &  \sum_{\ell < n} {{n_1}\choose{\ell_1}}\cdots{{n_k}\choose{\ell_k}} \prod_{i=1}^k |s_i|^{n_i-\ell_i} (\eta_n - \eta_{\ell})![\nbE(M)m_L e^k]^{\eta_n - \eta_{\ell}}
.\, U_{\l} \nonumber\\
&\le & \sum_{\ell < n} {{n_1}\choose{\ell_1}}\cdots{{n_k}\choose{\ell_k}} (\eta_n - \eta_{\ell})!\left[\frac{1}{2k}\right]^{\eta_n - \eta_{\ell}}
.\, U_{\l}, \label{upper_bound_conv_distrib2_NBU}
\end{eqnarray}
the last inequality coming from the fact that $\prod_{i=1}^k |s_i|^{n_i-\ell_i}[\nbE(M)m_L e^k]^{\eta_n - \eta_{\ell}}\le \left[\frac{1}{2k}\right]^{\eta_n - \eta_{\ell}}$ for all $s=(s_1,...,s_k)\in J'$. Since $U_{\l}=\eta_\ell ! \left[\frac{1}{2k}\right]^{\eta_{\ell}}$, the right-hand side of \eqref{upper_bound_conv_distrib2_NBU} is equal to
\begin{equation}\label{rhs_upper_bound_conv_distrib2_NBU}
\sum_{\ell < n} {{n_1}\choose{\ell_1}}\cdots{{n_k}\choose{\ell_k}} (\eta_n - \eta_{\ell})!\eta_\ell !\left[\frac{1}{2k}\right]^{\eta_n},
\end{equation}
which we need to prove is equal to $U_n$. For this we use the following representation of multinomial distributed random vectors. One has that a random vector $(A_1,...,A_k)$ follows a ${\cal M}(\eta_n, 1/k,...,1/k)$ distribution if and only if one has that $A_j=\sum_{i=1}^{\eta_n} \nbu_{[Y_i=j]}$, $j\in \{1,...,k\}$, where $Y_1$,..., $Y_{\eta_n}$ are iid and
uniformly distributed on the set $\{1,...,k\}$.
One thus deduces that the joint event $\left[ A_j=n_j,\ j=1,...,k\right]$ can be written as the union of disjoint sets as follows
\begin{multline*}\label{decompose_multinom}
\left[ A_j=n_j,\ j=1,...,k\right]= \left[ \sum_{i=1}^{\eta_n} \nbu_{[Y_i=j]}=n_j,\ j=1,...,k\right]\\
=\bigcup_{r=1}^{\eta_n-1} \left[ \sum_{i=1}^r \nbu_{[Y_i=j]}=\ell_j,\sum_{i=r+1}^{\eta_n} \nbu_{[Y_i=j]}=n_j-\ell_j,\ j=1,...,k,\ \mbox{for some }\ell=(\ell_1,...,\ell_k)\mbox{ s.t. }\eta_\ell=r\right]\\
= \bigcup_{r=1}^{\eta_n-1} \left[ B_j^r=\ell_j,C_j^{\eta_n-r}=n_j-\ell_j,\ j=1,...,k,\ \mbox{for some }\ell=(\ell_1,...,\ell_k)\mbox{ s.t. }\eta_\ell=r\right] ,
\end{multline*}
where $(B_1^r,...,B_k^r)_{r\in\{1,..., \eta_n-1\}}$ and $(C_1^{\eta_n-r},...,C_k^{\eta_n-r})_{r\in\{1,..., \eta_n-1\}}$ are defined by $B_j^r:=\sum_{i=1}^r \nbu_{[Y_i=j]}$, $C_j^{\eta_n-r}=\sum_{i=r+1}^{\eta_n} \nbu_{[Y_i=j]}$, $j\in\{1,...,k\}$, and are thus two iid and independent families of random vectors with $(B_1^r,...,B_k^r)\sim {\cal M}(r, 1/k,...,1/k)$ and $(C_1^{\eta_n-r},...,C_k^{\eta_n-r}) \sim {\cal M}(\eta_n-r, 1/k,...,1/k)$, $r\in\{1,..., \eta_n-1\}$. Let us introduce for all $r\in \nbN^*$ the set ${\cal A}_r:=\{n\in \nbN^k|\ \eta_n=r\}\subset \nbN^k\setminus \{{\bf 0} \}$. \eqref{rhs_upper_bound_conv_distrib2_NBU}  is then computed as follows
\begin{eqnarray*}
&&\sum_{\ell < n} {{n_1}\choose{\ell_1}}\cdots{{n_k}\choose{\ell_k}} (\eta_n - \eta_{\ell})!\eta_\ell !\left[\frac{1}{2k}\right]^{\eta_n}=\prod_{i=1}^k n_i! \frac{1}{2^{\eta_n}}\sum_{\ell < n} \frac{\eta_\ell !}{\prod_{i=1}^k \ell_i!} \left[\frac{1}{k}\right]^{\eta_\ell} \frac{(\eta_n - \eta_{\ell})!}{\prod_{i=1}^k (n_i-\ell_i)!}\left[\frac{1}{k}\right]^{\eta_n-\eta_\ell}\\
&=& \prod_{i=1}^k n_i! \frac{1}{2^{\eta_n}} \sum_{r=1}^{\eta_n-1}\sum_{\ell \in{\cal A}_r} \frac{r !}{\prod_{i=1}^k \ell_i!} \left[\frac{1}{k}\right]^{r} \frac{(\eta_n - r)!}{\prod_{i=1}^k (n_i-\ell_i)!}\left[\frac{1}{k}\right]^{\eta_n-r}\\
&=& \prod_{i=1}^k n_i! \frac{1}{2^{\eta_n}} \sum_{r=1}^{\eta_n-1}\sum_{\ell \in{\cal A}_r} \nbP\left[ B_j^r=\ell_j , j=1,...,k\right]\nbP\left[ C_j^{\eta_n-r}=n_j-\ell_j ,  j=1,...,k \right]\\
&=& \prod_{i=1}^k n_i! \frac{1}{2^{\eta_n}} \sum_{r=1}^{\eta_n-1}\sum_{\ell \in{\cal A}_r} \nbP\left[ B_j^r=\ell_j , \ C_j^{\eta_n-r}=n_j-\ell_j,\  j=1,...,k\right]\\
&=& \prod_{i=1}^k n_i! \frac{1}{2^{\eta_n}} \nbP[A_j=n_j,\ j=1,...,k]=\prod_{i=1}^k n_i! \frac{1}{2^{\eta_n}} \frac{\eta_n!}{\prod_{i=1}^k n_i!}\left[\frac{1}{k}\right]^{\eta_n}=U_n,
\end{eqnarray*}
which completes the induction. We now conclude this case with the fact that, similarly to \eqref{upper_bound_conv_distrib}, and thanks to \eqref{bound_M_tilde},
$$
\prod_{i=1}^k |s_i|^{n_i} \tilde{M}_n(t)\le C \prod_{i=1}^k |s_i|^{n_i} R_n \le C U_n,\quad \forall n\in \nbN^k,\ \forall s\in J',
$$
\begin{multline*}
\mbox{with }\sum_{n\in \nbN^k} \frac{U_n}{\prod_{i=1}^k n_i!}= \sum_{n\in \nbN^k}\frac{1}{2^{\eta_n}} \frac{\eta_n!}{\prod_{i=1}^k n_i!} \frac{1}{k^{\eta_n}}\\
= \frac{1}{2^{\eta_{\bf 0}}} \eta_{\bf 0}! \frac{1}{k^{\eta_{\bf 0}}} +  \sum_{n\in \nbN^k\setminus \{ {\bf 0}\}}\frac{1}{2^{\eta_n}} \frac{\eta_n!}{\prod_{i=1}^k n_i!} \frac{1}{k^{\eta_n}}
=1+\sum_{r=1}^\infty \frac{1}{2^r}\sum_{n \in{\cal A}_r} \frac{r!}{\prod_{i=1}^k n_i!} \frac{1}{k^{r}}.
\end{multline*}
Noting that for all $r\ge 1$, $\sum_{n \in{\cal A}_r} \frac{r!}{\prod_{i=1}^k n_i!} \frac{1}{k^{r}}= \sum_{n \in{\cal A}_r}\nbP[A_1=n_1,...,A_k=n_k]=1$ with a random vector $(A_1,...,A_k)\sim {\cal M}(r, 1/k,...,1/k)$ defined similarly as previously, we thus deduce that $\sum_{n\in \nbN^k} \frac{U_n}{\prod_{i=1}^k n_i!}=\sum_{r=0}^\infty \frac{1}{2^r}<+\infty$. Then, we conclude by the dominated convergence theorem that $\tilde{\psi}(s,t)\longrightarrow \psi_\infty(s)$ when $t\to\infty$ for $s\in J'$.

To sum up when $(X_1,...,X_n)$ satisfies {\bf (A2)}, since $\tilde{M}_n(t)$ and $\chi_n$ are bounded as shown in Proposition \ref{tildeMnRn}, the mgfs of $e^{\delta t}Z(t)$ in (\ref{mgfedtZt}) and $\mathcal{Z}_\infty$ in (\ref{mgfLinf}) exist. Also, we have shown that
$\tilde{\psi}(s,t)\longrightarrow \psi_\infty(s)$ when $t\to\infty$ for $s\in J$ or $J'$ in some neighborhood of $(0,\ldots,0)$. Hence,
$e^{\delta t}Z(t)$ converges to $\mathcal{Z}_\infty$ in distribution.

\subsection{Proof of Lemma \ref{lemma_recursive}}\label{proof_lemma_recursive}
When $n=n(i)$ and $i\in\{1,\ldots,k\}$,
we may obtain an expression of $\L^b_{n}(s)$ by using (\ref{bnt}), and applying similar idea as applied in (\ref{int_b_tilde}). We now turn to proving (\ref{recursion_LT_b}). Since $L_j$'s are all ${\cal E}(\mu)$ distributed, $\ph_{\l ,n}(t)$ given by (\ref{phi_l}) simplifies to
$$
\ph_{\l ,n}(t)= \nbE\bigg[ \tilde{M}_\ell(t-\tau_1)\bigg\{\prod_{j\in C_{\ell ,n}} \frac{\mu}{\mu+(n_j-\l_j)\delta} \bigg\} e^{-|C_{\l ,n}| \mu(t-\tau_1)}
.\nbu_{[\tau_1<t]}\bigg].
$$
Then using Fubini's theorem to interchange the expectation with the integration as well as a change of variable $t:=t-\tau_1$, it follows that
\begin{eqnarray}
\int_0^\infty e^{-ut}\ph_{\l ,n}(t)dt &=& \bigg[\prod_{j\in C_{\ell ,n}} \frac{\mu}{\mu+(n_j-\l_j)\delta} \bigg]\nbE\left[ \int_{\tau_1}^\infty e^{-ut} \tilde{M}_\ell(t-\tau_1)e^{-|C_{\l ,n}| \mu(t-\tau_1)} dt \right]\nonumber\\
&=&\bigg[\prod_{j\in C_{\ell ,n}} \frac{\mu}{\mu+(n_j-\l_j)\delta} \bigg]\nbE\left[ e^{-u\tau_1} \int_{0}^\infty e^{-ut} \tilde{M}_\ell(t)e^{-|C_{\l ,n}| \mu t} dt\right]\nonumber\\
&=& \bigg[\prod_{j\in C_{\ell ,n}} \frac{\mu}{\mu+(n_j-\l_j)\delta} \bigg]\L^\tau (u) \L^M_\l(u+|C_{\l ,n}| \mu).\label{compute_LT_recursive}
\end{eqnarray}
If $\ell=\mathbf{0}$, then $\tilde{M}_\ell(t)=1$ and thus $ \L^M_\l(u+|C_{\mathbf{0},n}| \mu)=\frac{1}{u+|C_{\mathbf{0},n}| \mu}$. Then, we get
$$
\int_0^\infty e^{-ut}\varphi_{\mathbf{0}, n}(t)dt =\bigg[\prod_{j\in C_{\mathbf{0},n}} \frac{\mu}{\mu+(n_j-\l_j)\delta} \bigg] \frac{\L^\tau (u)}{u+ |C_{\mathbf{0},n}|\mu}=\bigg[\prod_{j=1}^k \frac{\mu}{\mu+n_j\delta} \bigg] \frac{\L^\tau (u)}{u+ |C_{\mathbf{0},n}|\mu}.
$$
When $\ell>\mathbf{0}$, let us now observe
that \eqref{LT_M_b} and (\ref{compute_LT_recursive}) lead to
$$
\int_0^\infty e^{-ut}\ph_{\l ,n}(t)dt=\bigg[\prod_{j\in C_{\ell ,n}} \frac{\mu}{\mu+(n_j-\l_j)\delta} \bigg]\frac{\L^\tau (u)}{1-\L^\tau (u+|C_{\l ,n}| \mu)}\, \L^b_\l (u+|C_{\l ,n}| \mu).
$$
With the above result, the Laplace transform of (\ref{exp_tilde_b}) becomes (\ref{recursion_LT_b}).

\subsection{Proof of Theorem \ref{theorem_expansion}}\label{proof_theorem_expansion}
Substituting (\ref{vx}) into (\ref{solMnt}) for $dm(s)$ yields
\[
\tilde{M}_1(t)=\frac{1}{\nbE[\tau_1]}\int^t_0 \tilde{b}_1(t-s)ds+\int^t_0 \tilde{b}_1(t-s)dv(x).
\]
A change of variable $s:=t-s$ in the first integral and a subtraction of $\chi_1$ in (\ref{chind}) on both sides result in
\begin{equation}\label{tMntminuscn}
\tilde{M}_1(t)-\chi_1 =-\frac{1}{\nbE[\tau_1]}\int^\infty_t \tilde{b}_1(s)ds+\int^t_0 \tilde{b}_1(t-s)dv(x).
\end{equation}
Let
\begin{equation}\label{I12t}
I_1(t)=-\frac{1}{\nbE[\tau_1]}\int^\infty_t \tilde{b}_1(s)ds,\qquad I_2(t)=\int^t_0 \tilde{b}_1(t-s)dv(s),
\end{equation}
then (\ref{tMntminuscn}) is essentially a sum of $I_1(t)$ and $I_2(t)$. In the sequel, we shall separately study the asymptotic behaviors of $I_1(t)$ and $I_2(t)$ when $t\rightarrow \infty$. First it is convenient to introduce the following quantity and its asymptotic result as it will be often utilized in the later analysis.
\begin{align}\label{useful}
\nbE[\nbu_{\{ \tau_1 \geq t\}}e^{-\mu_i(t-\tau_1)}]
& = e^{-\mu_i t}\int^\infty_t e^{\mu_i s}dF(x)\nonumber
\\
&= e^{-\mu_i t} \int^\infty_t e^{(\mu_i-R)s} e^{Rs} dF(s) \leq e^{-\mu_i t} \int^\infty_t e^{(\mu_i-R)t} e^{Rs} dF(s)\nonumber\\
&\leq e^{-Rt} \int^\infty_t e^{Rs} dF(x)=o (e^{-R t}),
\end{align}
where the second last inequality is due to the assumption on $\mu_i <R$ for all $i$'s and the last result is due to $\nbE[e^{R\tau_1}]=\L^\tau(-R)<\infty$ by {\bf (A1')}.

We begin to analyze $I_1(t)$ in (\ref{I12t}) when $t\rightarrow \infty$. From (\ref{bnt}) and (\ref{bomega}) we may write
\begin{equation}\label{I1t}
\int^\infty_t \tilde{b}_1(z)dz= \nbE[X].
\nbE\Big[ \int^\infty_t e^{\delta(z-\tau_1)}
\nbu_{\{ \tau_1 <z\}} \int^\infty_{z-\tau_1}
e^{-\delta s} dW(s)dz\Big].
\end{equation}
When we assume that $L_j$'s are ${\cal E}(\mu)$ distributed for $\mu>0$, then the second integral on the above equation is simplified as
\begin{equation}\label{abc}
\int^\infty_{z-\tau_1}
e^{-\delta s} dW(s)=\frac{\mu}{\mu+\delta} e^{-(\mu+\delta)(z-\tau_1)}.
\end{equation}
As $\nbu_{\{ \tau_1 \geq t\}}+\nbu_{\{ \tau_1 < t\}}=1$, inserting these two indicator functions in (\ref{I1t}) together with (\ref{abc}) results in
\begin{equation*}\label{I1ta}
\int^\infty_t \tilde{b}_1(z)dz
=\nbE[X_i].\frac{\mu}{\mu+\delta}\,\nbE\Big[
\left(\nbu_{\{ \tau_1 < t\}}+\nbu_{\{ \tau_1 \geq t\}}\right)
\int^\infty_t 
\nbu_{\{ \tau_1 <z\}} e^{-\mu (z-\tau_1)}dz\Big].
\end{equation*}
For the case of $\tau_1 < t$, as $z>t$ and $ \tau_1 < z$, the above expectation is reduced to
\begin{align*}\label{aaa}
\nbE\Big[
\nbu_{\{ \tau_1 < t\}}
\int^\infty_t 
\nbu_{\{ \tau_1 <z\}} e^{-\mu(z-\tau_1)}dz\Big]&=\frac{1}{\mu}\nbE[
\nbu_{\{ \tau_1 < t\}}  e^{-\mu t-\tau_1)}]\nonumber\\
&=\frac{1}{\mu}\nbE[
(1-\nbu_{\{ \tau_1 \geq t\}})  e^{-\mu(t-\tau_1)}]\nonumber\\
&= \frac{1}{\mu} \left\{e^{-\mu t} \L^{\tau} (-\mu)
- \nbE[ \nbu_{\{ \tau_1 \geq t\}}  e^{-\mu(t-\tau_1)}] \right\}\nonumber\\
&=\frac{1}{\mu} e^{-\mu t} \L^{\tau} (-\mu)+o (e^{-R t}),
\end{align*}
where the last line is obtained by applying (\ref{useful}).
On the other hand, when $\tau_1 \geq t$,
\begin{align*}
\nbE\Big[
\nbu_{\{ \tau_1 \geq  t\}}
\int^\infty_t 
\nbu_{\{ \tau_1 <z\}} e^{-\mu(z-\tau_1)}dz\Big] &=\nbE\Big[
\nbu_{\{ \tau_1 \geq t\}}
\int^\infty_{\tau_1}  e^{-\mu (z-\tau_1)}dz\Big]
\nonumber\\
&=\frac{1}{\mu} \nbP(\tau_1\geq t),
\end{align*}
and note that, using Chernoff's inequality, $\nbP(\tau_1 \geq t)\leq \nbE(e^{R\tau_1}) e^{-Rt} = o(e^{-z_N t})$ because of $\EE(e^{R\tau_1})<\infty$ (by condition {\bf (A1')}) and $\mathrm{Re}(z_N)<R$. Hence combining the above results using the fact that an $o(e^{-Rt})$ is a fortiori an $o(e^{-z_Nt})$, it follows that
\begin{equation}\label{I1tb}
I_1(t)=
-\frac{1}{\nbE[\tau_1]}\int^\infty_t \tilde{b}_1(s)ds=
-\frac{\nbE[X]}{\nbE[\tau_1]}.\frac{1}{\mu+\delta} \L^{\tau}(-\mu)e^{-\mu t}+ o(e^{-z_N t}).
\end{equation}
We now turn to $I_2(t)$ in (\ref{I12t}). As $\tilde{b}_1(0)=0$, applying integration by parts for Stieltjes integrals on the right-hand side of $I_2(t)$ yields
\begin{equation}\label{I2t1}
I_2(t)= \int^t_0 \tilde{b}_1(t-s)dv(x)=\tilde{b}_1(t) v(0^-)+\int^t_0 v(s) \tilde{b}_1'(t-s)ds.
\end{equation}
But $v(0^-)=-\nbE[\tau_1^2]/(2\nbE[\tau_1]^2)$
and using a similar reasoning applied to (\ref{useful}) we get
\begin{equation}\label{btnto}
\tilde{b}_1(t)=\nbE[X].\frac{\mu}{\mu+\delta}
\nbE\big[
\nbu_{\{ \tau_1 <t\}}
e^{-\mu (t-\tau_1)}\big]=\nbE[X].\frac{\mu}{\mu+\delta}\L^{\tau} (-\mu)e^{-\mu t} +o (e^{-R t}),
\end{equation}
i.e.
\begin{equation}\label{bnv0}
\tilde{b}_1(t)v(0^-)=-\frac{\nbE[X]\nbE[\tau_1^2]}{
2\nbE[\tau_1]^2}\frac{\mu}{\mu+\delta}\L^{\tau} (-\mu)e^{-\mu t} +o (e^{-z_N t}),\qquad t\rightarrow \infty.
\end{equation}
Also we have $\tilde{b}_1(t)=\nbE[X].\frac{\mu}{\mu+\delta}e^{-\mu t}\int^t_0 e^{\mu s}dF(s)$ and then
$
\tilde{b}_1'(t)=-\mu \tilde{b}_1(t)+\nbE[X]
\frac{\mu}{\mu+\delta}f(t).
$
Thus
\begin{align}\label{one}
\int^t_0 e^{-z_k s} \tilde{b}_1'(t-s)dx
&= e^{-z_k t}\int^t_0 e^{z_k s} \tilde{b}_1'(s)ds\nonumber\\
&=  e^{-z_k t}  \int^t_0 e^{z_k s} \Big[-\mu\tilde{b}_1(s)+\nbE[X] \frac{\mu}{\mu+\delta}f(s)
\Big]ds , \quad k=1,\ldots,N.
\end{align}
On the first term of the above equation, from (\ref{btnto}) it follows that
\begin{align}\label{two}
e^{-z_k t}  \int^t_0 e^{z_k s}
\tilde{b}_1(s)ds &=\nbE[X].\frac{\mu}{\mu+\delta}
\Big(\frac{1}{z_k-\mu}\Big)\nbE\Big[ \nbu_{\{ \tau_1 < t\}} \{e^{-\mu(t-\tau_1)}- e^{-z_k(t-\tau_1)}\}\Big]\nonumber\\
&= \nbE[X].\frac{\mu}{\mu+\delta}
\Big(\frac{1}{z_k-\mu}\Big)
\left\{e^{-\mu t} \L^{\tau} (-\mu)-e^{-z_k t} \L^{\tau}(-z_k)\right\}+o (e^{-R t}),
\end{align}
for $k=1,\ldots,N$. Next, on the second term, one has
\begin{align}\label{three}
e^{-z_k t}\int^t_0 e^{z_k s} f(s)ds
&=e^{-z_k t} \L^\tau(-z_k)-e^{-z_k t}\int^\infty_t
e^{z_k s} f(s)ds\nonumber\\
&= e^{-z_k t} \L^\tau(-z_k)+o(e^{-z_N t})
\end{align}
since
\begin{align*}
\Big|e^{-z_k t} \int^\infty_t
e^{z_k s} f(s)ds\Big| &= \Big|e^{-z_k t} \int^\infty_t e^{(z_k-R) s} e^{Rs}f(s)ds\Big|
\leq e^{-\mathrm{Re}(z_k) t} \int^\infty_t e^{(\mathrm{Re}(z_k)-R) s} e^{Rs}f(s)ds\\
&\leq  e^{-\mathrm{Re}(z_k) t}e^{(\mathrm{Re}(z_k)-R) t}\int^\infty_t
 e^{Rs}f(s)ds = e^{-Rt} \int^\infty_0 e^{Rs}f(s)ds=o(e^{-z_N t}).
\end{align*}
Then using (\ref{vxexpan}) and (\ref{one}) with (\ref{two}) and (\ref{three}), and since an $o(e^{-Rt})$ is a fortiori an $o(e^{-z_Nt})$, the second term of (\ref{I2t1}) (except for the term involving $o (e^{-z_N x})$ in $v(x)$ in (\ref{vxexpan})) is now given by
\begin{align}\label{intvbp}
&\int^t_0 [v(s)-o (e^{-z_N s})] \tilde{b}_1'(t-s)ds \nonumber\\
&~~= \nbE[X].\frac{\mu}{\mu+\delta}\left[ \sum^N_{k=1}\gamma_k
\Big(\frac{\mu}{z_k-\mu}\Big)
\left\{e^{-z_k t} \L^{\tau}(-z_k)-e^{-\mu t} \L^{\tau} (-\mu)\right\}+\gamma_k e^{-z_k t} \L^\tau(-z_k)\right]+o(e^{-z_N t})\nonumber\\
&~~=\nbE[X].\frac{\mu}{\mu+\delta}\left[ \sum^N_{k=1}\gamma_k \left(\frac{z_k}{z_k-\mu}e^{-z_k t} \L^{\tau}(-z_k)-\frac{\mu}{z_k-\mu}e^{-\mu t} \L^{\tau} (-\mu)\right)\right]+o(e^{-z_N t}).
\end{align}
Recall that function $\eta(.)$ is defined by (\ref{function_eta}). Then, putting the expression for $\tilde{b}_1'(t)$ into the integral, it follows that
\begin{align}\label{intobp}
\int^t_0 o(e^{-z_N s})\tilde{b}_1'(t-s)ds
&=\int^t_0  \eta(s)e^{-z_N s}\tilde{b}_1'(t-s)ds \nonumber\\
&=\int^t_0 \eta(s)e^{-z_N s} \Big[
-\mu \tilde{b}_1(t-s)+\nbE[X]
\frac{\mu}{\mu+\delta}f(t-s)\Big] ds.
\end{align}
We start by considering $\int^t_0 \eta(s)e^{-z_N s}
f(t-s)ds$ which can be written as
\[
\int^t_0 \eta(t-s)e^{-z_N (t-s)}
f(s)ds=e^{-z_N t} \int^\infty_0 \eta(t-s)\nbu_{\{ 0<s< t\}}e^{ z_N s} f(s)ds.
\]
The fact that $\int^\infty_0 |e^{z_N s} f(s)ds|
=\int^\infty_0 e^{(\mathrm{Re}(z_N))s} f(s)ds$ is convergent implies, by the dominated convergence theorem,
\[
\int^\infty_0 \eta(t-s)\nbu_{\{ 0<s< t\}}e^{ z_N s} f(s)ds \longrightarrow 0, \qquad t\rightarrow \infty.
\]
Consequently,
\begin{equation}\label{intetaf}
\int^t_0 \eta(s)e^{-z_N s}
f(t-s)ds=o(e^{-z_N t}),\qquad t\rightarrow \infty.
\end{equation}
Now we turn our attention to the first term of (\ref{intobp}) involving $\int^t_0 \eta(s)e^{-z_N s}\tilde{b}_1(t-s)ds$.
Writing from (\ref{bnt}) (see also (\ref{btnto}))
$$
\tilde{b}_1(t)=\nbE[X].\frac{\mu}{\mu+\delta}
\nbE\big[
\nbu_{\{ \tau_1 <t\}}
e^{-\mu (t-\tau_1)}\big]= \nbE[X].\frac{\mu}{\mu+\delta}  {\cal L}^\tau(-\mu) e^{-\mu t}- \nbE[X].\frac{\mu}{\mu+\delta} \nbE\big[
\nbu_{\{ \tau_1 \ge t\}}
e^{-\mu (t-\tau_1)}\big],
$$
we then split $\int^t_0 \eta(s)e^{-z_N s}\tilde{b}_1(t-s)ds$ into two parts, namely $\nbE[X].\frac{\mu}{\mu+\delta}  {\cal L}^\tau(-\mu)\int_0^t \eta(s) e^{-z_N s} e^{-\mu(t-s)}ds$ and $\nbE[X].\frac{\mu}{\mu+\delta}   \int_0^t \eta(s) e^{-z_N s} \nbE\big[
\nbu_{\{ \tau_1 \ge t-s\}}
e^{-\mu ((t-s)-\tau_1)}\big]  ds$. The first term is expressed as
\begin{align}\label{eta_expansion1}
&\nbE[X].\frac{\mu}{\mu+\delta}  {\cal L}^\tau(-\mu)\int_0^t \eta(s) e^{-z_N s} e^{-\mu(t-s)}ds
\nonumber\\
&~~=\nbE[X].\frac{\mu}{\mu+\delta}  {\cal L}^\tau(-\mu)\left[\int_0^\infty \eta(s) e^{-z_N s} e^{\mu s}ds\right] e^{-\mu t}- \nbE[X].\frac{\mu}{\mu+\delta}  {\cal L}^\tau(-\mu)\left[\int_t^\infty \eta(s) e^{-z_N s} e^{\mu s}ds\right] e^{-\mu t}\nonumber\\
&~~= \nbE[X].\frac{\mu}{\mu+\delta}  {\cal L}^\tau(-\mu)\left[\int_0^\infty \eta(s) e^{-z_N s} e^{\mu s}ds\right] e^{-\mu t} + o(e^{-z_N t}),
\end{align}
where the latter term $o(e^{-z_N t})$ being again justified as in (\ref{useful}). Now (\ref{useful}) implies that the second term verifies, by the dominated convergence theorem
\begin{align}\label{eta_expansion2}
&\nbE[X].\frac{\mu}{\mu+\delta}   \int_0^t \eta(s) e^{-z_N s} \nbE\big[
\nbu_{\{ \tau_1 \ge t-s\}}
e^{-\mu ((t-s)-\tau_1)}\big]  ds\nonumber\\
&~~=\nbE[X].\frac{\mu}{\mu+\delta}   e^{-z_N t}\int_0^t \eta(t-s) e^{z_N s} \nbE\big[
\nbu_{\{ \tau_1 \ge s\}}
e^{-\mu (s-\tau_1)}\big]  ds=o(e^{-z_N t}).
\end{align}
Gathering (\ref{eta_expansion1}) and (\ref{eta_expansion2}) thus yields
\begin{equation}\label{intetab}
\int^t_0 \eta(s)e^{-z_N s}\tilde{b}_1(t-s)ds = \nbE[X].\frac{\mu}{\mu+\delta}\L^{\tau} (-\mu) \Big[\int^\infty_0 \eta(s)e^{(\mu-z_N)s} ds \Big]e^{-\mu t}+o(e^{-z_N t}).
\end{equation}
Then from (\ref{intvbp}) and (\ref{intobp}) with (\ref{intetaf}) and (\ref{intetab}) we get
\begin{align*}
\int^t_0 v(s) \tilde{b}_1'(t-s)ds&=
\nbE[X].\frac{\mu}{\mu+\delta}\left[ \sum^N_{k=1}\gamma_k \left(\frac{z_k}{z_k-\mu}e^{-z_K t} \L^{\tau}(-z_K)-\frac{\mu}{z_k-\mu}e^{-\mu t} \L^{\tau} (-\mu)\right)\right]\\
&~~-\nbE[X].\frac{\mu^{2}}{\mu+\delta}\L^{\tau} (-\mu) \Big[\int^{ \infty}_0 \eta(s)e^{(\mu-z_N)s} ds \Big]e^{-\mu t}+o(e^{-z_N t}),\qquad t\rightarrow \infty.
\end{align*}
Hence the above result together with (\ref{bnv0}) allows us to have an expression for (\ref{I2t1}) as
\begin{equation}\label{I2t}
I_2(t)=A e^{-\mu t}+\sum^N_{k=1} B_{k}e^{-z_k t}
+o(e^{-z_N t}),
\end{equation}
where $A$ and $B_{k}$ for $k=1,\ldots,N$ are defined by (\ref{Ai}) and (\ref{Bki}). As a result, combining (\ref{I1tb}) and (\ref{I2t}) leads to the theorem.

\subsection{Proof of Theorem \ref{workload}}\label{proof_workload}
{\bf Proof of Lemma \ref{Lemma_analytic}}. We shall start by proving the properties for $\tilde{M}_{1}(t,\delta)$, as those for $\tilde{M}_{2}(t,\delta)$ are a bit more technical but follow in a similar way. Let us write
\begin{equation}\label{analytic_1}
\tilde{M}_{1}(t,\delta)=\sum_{i=1}^\infty \psi_i(t,\delta),\quad \psi_i(t,\delta):= \EE[e^{-\delta(T_i+L_i-t)}\nbu_{\{ T_i \le t< T_i+L_i\}}],\qquad i\in\nbN.
\end{equation}
We first start by proving that $\psi_i(t,\delta)$ is defined and analytic on set $D_a$. Indeed, inequality
\begin{equation}\label{analytic_2}
\left| \delta^j \frac{(-1)^j}{j!}(T_i+L_i-t)^j \nbu_{\{ T_i \le t< T_i+L_i\}}\right|\le a^j \frac{1}{j!} L_i^j,\quad j\in \nbN,\quad \delta \in D_a,
\end{equation}
coupled with the fact that $\sum_{j=0}^\infty \EE\left[ a^j \frac{1}{j!} L_i^j\right]= \EE[e^{a L}]=\frac{\mu}{\mu-a}<+\infty$ by (\ref{assumption_light_tailed}), yields that 
\[\sum_{j=0}^\infty \delta^j \EE \left[ \frac{(-1)^j}{j!}(T_i+L_i-t)^j \nbu_{\{ T_i \le t< T_i+L_i\}}\right]
\]
 is a convergent series on $\delta\in D_a$ and that $\delta\mapsto \psi_i(t,\delta)$ is analytic on that set for all $t\ge 0$. Also, $\psi_i(t,\delta)$ admits the above power series expansion in $\delta$. Now one checks easily, by independence of $L_i$ and $T_i$,
\begin{equation}\label{analytic_3}
\psi_i(t,\delta)\le \EE[e^{a L_i}\nbu_{\{T_i\le t\}}]= \EE[e^{a L}]\nbP[T_i\le t],\quad  \forall \delta\in D_a,
\end{equation}
with $\sum_{i=1}^\infty \EE[e^{a L}]\nbP[T_i\le t] = \EE[e^{a L}] m(t)<+\infty$. This yields that for all $t\ge 0$, series $\sum_{i=1}^\infty\psi_i(t,\delta)$ converges normally on  $\delta \in D_a$. Thus for all $t\ge 0$, $\delta\mapsto \tilde{M}_{1}(t,\delta)$ is analytic as the uniform limit of an analytic sequence of functions on compact set $D_a$.

We then move on $\tilde{M}_{2}(t,\delta)$. Similar to (\ref{analytic_1}), one has
\begin{equation*}\label{analytic_4}
\tilde{M}_{2}(t,\delta)=\sum_{r,j=1}^\infty \pi_{r,j}(t,\delta),\quad \pi_{r,j}(t,\delta):=\EE[e^{-\delta(T_r+L_r-t)}\nbu_{\{ T_r \le t< T_r+L_r\}}e^{-\delta(T_j+L_j-t)}\nbu_{\{ T_j \le t< T_j+L_j\}}].
\end{equation*}
The analog of (\ref{analytic_2}) is
\begin{multline*}\label{analytic_5}
\left| \delta^p \frac{(-1)^p}{p!}[(T_r+L_r-t)+(T_j+L_j-t)]^p\, \nbu_{\{ T_r \le t< T_r+L_r\}}\nbu_{\{ T_j \le t< T_j+L_j\}}\right|\le (a/2)^p \frac{1}{p!} [L_r+L_j]^p,\\
 r\in \nbN,\ j\in \nbN,\quad \delta\in D_{a/2},
\end{multline*}
with $\sum_{p=0}^\infty (a/2)^p \frac{1}{p!} [L_r+L_j]^p = \EE\left[ e^{a(L_r+L_j)/2}\right] \le \EE\left[ e^{aL}\right]$ (by Jensen's inequality), a finite quantity, so that $\delta\in D_{a/2} \mapsto \pi_{r,j}(t,\delta)$ is analytic. The analog of (\ref{analytic_3}) is
\begin{equation}\label{analytic_6}
\pi_{r,j}(t,\delta)\le \EE\left[e^{a (L_r+L_j)/2}\nbu_{\{T_r\le t\}}\nbu_{\{T_j\le t\}}\right],\quad r\in \nbN,\ j\in \nbN,\ \delta\in D_{a/2},
\end{equation}
with, again thanks to Jensen's inequality as well as independence of $(L_r,L_j)$ from $(T_r,T_j)$,
$$\sum_{r,j=1}^\infty \EE[e^{a (L_r+L_j)/2}\nbu_{\{T_r\le t\}}\nbu_{\{T_j\le t\}}]\le \EE\left[ e^{aL}\right]\sum_{r,j=1}^\infty \EE\left[ \nbu_{\{T_r\le t\}}\nbu_{\{T_j\le t\}}\right]=\EE\left[ e^{aL}\right] \EE\left[ N_t^2\right]<+\infty.$$
Hence, from (\ref{analytic_6}), $ \sum_{r,j=1}^\infty \pi_{r,j}(t,\delta)= \tilde{M}_{2}(t,\delta)$ converges normally on $\delta\in D_{a/2}$, and is analytic on this set by the same argument as $\delta\mapsto\tilde{M}_{1}(t,\delta)$. Note that we used the fact that $N_t$ admits the second moment, due to $\EE[\tau_1^2]<+\infty$, see e.g. \cite[Chapter V.6]{APQ}.\hfill $\Box$

Prior to proving Lemma \ref{lemma_uniform_convergence}, we find some upper bounds concerning $\tilde{M}_{1}(t,\delta)$. First, we note that deriving $\tilde{b}_{1}(t)=\frac{\mu}{\mu+\delta}e^{-\mu t}\int^t_0 e^{\mu s}dF(s)$ yields $\tilde{b}_{1}'(t)=-\mu \tilde{b}_{1}(t)+ \frac{\mu}{\mu+\delta}f(t)$. Besides, since {\bf (A1)} holds, a density $u(t)=m'(t)$ of renewal function exists and is bounded by above by $C>0$ thanks to Lemma \ref{lemma_density_upper_bound}. Both these facts entail, deriving (\ref{solMnt}), the following
\[
\left|\tilde{M}_{1}'(t)\right|=
\left|\int^t_0 \tilde{b}_{1}'(t-s)m'(s)ds + \tilde{b}_{1}(0)m'(t)\right|=\left|\int^t_0 \tilde{b}_{1}'(t-s)m'(s)ds \right|
\]
as $f(.)$ is a density, so that $\tilde{b}_{1}(0)=0$.
Then one finds
\begin{eqnarray}
\left|\tilde{M}_{1}'(t)\right|
&\le & \mu  \int^t_0 \left|\tilde{b}_{1}(t-s)m'(s)\right| ds + \left|\frac{\mu}{\mu+\delta}\right| \int^t_0 \left|f(t-s)m'(s)\right| ds
\nonumber\\
&\le & \mu C \int^\infty_0  \left|\tilde{b}_{1}(s) \right|ds + \left|\frac{\mu}{\mu+\delta}\right| C \int^\infty_0 f(s)ds \nonumber\\
&\le &  C \left[ \left|\frac{\mu}{\mu+\delta}\right| +\left|\frac{\mu}{\mu+\delta}\right|\right]\le \frac{2C\mu}{\mu-|\delta|},
 \label{bound_tildeM_n(1)}
\end{eqnarray}
where the last line is due to the fact that $f(\cdot)$ is a density, and $\int^\infty_0 |\tilde{b}_{1}(s)|ds\le C\big|\frac{\mu}{\mu+\delta}\big|$ from (\ref{int_b_tilde}).

\noindent{\bf Proof of Lemma \ref{lemma_uniform_convergence}.} We again start with $\tilde{M}_{1}(t,\delta)$.
 The key is to use expansions for $\tilde{M}_{1}(t)=\tilde{M}_{1}(t,\delta)$ in Theorem \ref{theorem_expansion} and particularly the dependence of this expansion in $\delta$ as discussed in Remark \ref{rem_dep_delta}. Indeed, an immediate consequence of (\ref{dep_delta1}) and (\ref{dep_delta2}) in Remark \ref{rem_dep_delta} is that
\begin{eqnarray*}
\left| \tilde{M}_{1}(t,\delta)- \chi_{1}(\delta)\right| &\le& \frac{M^\ast}{\mu-|\delta|}\left[ e^{-\mu t}+ \sum_{k=1}^N e^{- \mathrm{Re}(z_k) t } + \zeta(t) e^{-\mathrm{Re}(z_N)t}\right]\\
&\le& \frac{M^\ast}{\mu-a}\left[ e^{-\mu t}+ \sum_{k=1}^N e^{- \mathrm{Re}(z_k) t } + \zeta(t) e^{-\mathrm{Re}(z_N)t}\right],\quad \forall \delta\in D_a,
\end{eqnarray*}
for some constant $M^\ast$ independent from $\delta$ and $t$, which implies the uniform convergence of $\tilde{M}_{1}(t,\delta)$ as $t\to\infty$ towards $\chi_{1}(\delta)$ on $\delta \in D_{a}$.

We then move on to $\tilde{M}_{2}(t,\delta)$. Relation (\ref{exp_tilde_b}) when $k=1$, $X_{j}=1$, $L\sim {\cal E}(\mu)$, along with (\ref{phi_l}) and (\ref{bomega}) yields the following expression
\begin{align}
\tilde{b}_{2}(t)=\tilde{b}_{2}(t,\delta)&= \ph_0(t,\delta) + 2\ph_1(t,\delta),\label{expr_bomega_workload}\\
\ph_0(t,\delta)= \ph_{0,2}(t,\delta)&= \frac{\mu}{\mu+2\delta}\nbE[e^{-\mu (t-\tau_1)}\nbu_{\{\tau_1<t\}}]=\frac{\mu}{\mu+2\delta}\int_0^t e^{-\mu (t-s)}f(s) ds,\label{expr_phi0_workload}\\
\ph_1(t,\delta)=\ph_{1, 2}(t,\delta) &= \frac{\mu}{\mu+\delta}\nbE[\tilde{M}_{1}(t-\tau_1,\delta) e^{-\mu (t-\tau_1)}\nbu_{\{\tau_1<t\}}]=\frac{\mu}{\mu+\delta}\int_0^t \tilde{M}_{1}(t-s,\delta)e^{-\mu (t-s)}f(s) ds.\label{expr_phi1_workload}
\end{align}
Differentiating (\ref{expr_phi0_workload}) and (\ref{expr_phi1_workload}) with respect to $t$ results in
\begin{eqnarray}
\ph_0'(t,\delta)&=&  \frac{\mu}{\mu+2\delta}\left[-\mu \int_0^t e^{-\mu (t-s)}f(s) ds + f(t)\right],\label{expr_phi0'_workload}\\
\ph_1'(t,\delta)  &=&  \frac{\mu}{\mu+\delta} \left[ \int_0^t \left(\tilde{M}_{1}'(t-s,\delta)-\mu\tilde{M}_{1}(t-s,\delta)\right)e^{-\mu (t-s)}f(s) ds + f(t)\right].\label{expr_phi1'_workload}
\end{eqnarray}
For later use we need to find upper bounds for $ \ph_0(t,\delta)$ and $\ph_1(t,\delta)$. Note that, since $\tilde{M}_{1}(t,\delta) $ converges uniformly on $\delta \in D_{a/2}$ as $t\to\infty$, it is uniformly bounded in $t\ge 0$ and $\delta \in D_{a/2}$ by some constant $\tilde{C}$. Therefore, one finds that (\ref{expr_phi0_workload}) and (\ref{expr_phi1_workload}) have upper bounds given by
\begin{eqnarray}
|\ph_0(t,\delta)| &\le & \left|\frac{\mu}{\mu+2\delta}\right| e^{-\mu t} \int_0^\infty e^{\mu s}f(s) ds \le \frac{\mu}{\mu-a} C_0 e^{-\mu t}, \quad \delta \in D_{a/2}, \label{bound_phi0_workload}\\
|\ph_1(t,\delta)| &\le & \left|\frac{\mu}{\mu+\delta}\right| \tilde{C} e^{-\mu t} \int_0^\infty e^{\mu s}f(s) ds \le \frac{\mu}{\mu-a/2} C_1 e^{-\mu t}, \quad \delta \in D_{a/2},\nonumber
\end{eqnarray}
for some constants $C_0$ and $C_1$ independent from $\delta \in D_{a/2}$ and $t$. We also wish to obtain similar bounds for $ \ph_0'(t,\delta)$ and $\ph_1'(t,\delta)$. The following upper bound for $ \ph_0'(t,\delta)$ is easily obtained thanks to (\ref{expr_phi0'_workload}):
\begin{equation} \label{bound_phi0'_workload}
|\ph_0'(t,\delta)| \le  \left|\frac{\mu}{\mu+2\delta}\right| \left[ \mu e^{-\mu t}\int_0^\infty e^{\mu s}f(s) ds + f(t)\right]\le \frac{\mu}{\mu-a } [C_0^\ast e^{-\mu t} + f(t) ],\quad \delta \in D_{a/2},
\end{equation}
for some constant $C_0^\ast$. As to $\ph_1'(t,\delta)$, recall that $t\mapsto \tilde{M}_{1}(t,\delta) $ and $t\mapsto \tilde{M}_{1}'(t,\delta) $ are uniformly bounded in $\delta \in D_{a/2}$ respectively by $\tilde{C}$ and $2C \frac{\mu}{\mu-a/2}$ (thanks to (\ref{bound_tildeM_n(1)})), then one easily finds from (\ref{expr_phi1'_workload})
\begin{equation*} \label{bound_phi1'_workload}
|\ph_1'(t,\delta)| \le \frac{\mu}{\mu-a/2} [C_1^\ast e^{-\mu t} + f(t) ],\quad \delta \in D_{a/2},
\end{equation*}
for some constant $C_1^\ast>0$. Getting back to our original concern of showing that $\tilde{M}_{2}(t,\delta)$ converges uniformly, we first note that $\tilde{M}_{2}(t,\delta)$ can also be expressed as (\ref{solMnt}) but with $\tilde{b}_2(t)$ in (\ref{expr_bomega_workload}) instead of $\tilde{b}_1(t)$. Then, to obtain the result as (\ref{chind}), from (\ref{expr_bomega_workload}) it is necessary and sufficient to prove that
$$
\delta\mapsto\int^t_0 \ph_l(t-s,\delta)dm(s), \quad l=0,1,
$$
converges uniformly on $\delta \in D_{a/2}$ as $t\to\infty$ towards $\frac{1}{\nbE[\tau_1]}\int^\infty_0 \ph_l(s,\delta)ds$ for $l=0,1$. Details will be given only for $l=0$ as similar proof is applicable for $l=1$. The starting point is the following decomposition, already used in Relation (\ref{tMntminuscn}) in Section \ref{proof_theorem_expansion}:
\begin{eqnarray}\label{decomposition_unif_convergence}
\int^t_0 \ph_0(t-s,\delta)dm(s)- \frac{1}{\nbE[\tau_1]}\int^\infty_0 \ph_0(s,\delta)ds &=&-\frac{1}{\nbE[\tau_1]}\int^\infty_t \ph_0(s,\delta)ds+\int^t_0\ph_0(t-s,\delta)dv(x)\nonumber\\
&:=& I_1(t,\delta)+I_2(t,\delta).
\end{eqnarray}
Thus, in view of (\ref{decomposition_unif_convergence}), it suffices to prove that $ I_1(t,\delta)$ and $ I_2(t,\delta)$ uniformly converge towards $0$ as $t\to\infty$ on $\delta \in D_{a/2}$. Uniform convergence of $ I_1(t,\delta)$ is obtained thanks to (\ref{bound_phi0_workload}) that entails:
$$
\sup_{\delta \in D_{a/2}}|I_1(t,\delta)|\le \frac{1}{\nbE[\tau_1]}\frac{1}{\mu-a} C_0 e^{-\mu t}\longrightarrow 0,\quad t\to\infty .
$$
As to $I_2(t,\delta)$, performing an integration by parts as in (\ref{I2t1}) yields
$$
I_2(t,\delta)=\ph_0(t,\delta) v(0^-)+\int^t_0 v(s) \ph_0'(t-s,\delta)ds.
$$
The first term on the right-hand side uniformly converges to $0$ on $\delta \in D_{a/2}$ thanks to (\ref{bound_phi0_workload}). As to the second term, we use the inequality (\ref{bound_phi0'_workload}) to get
\begin{equation}\label{ineg_I_2}
\left|\int^t_0 v(s) \ph_0'(t-s,\delta) ds\right| \le \int^t_0 |v(s)| |\ph_0'(t-s,\delta)| ds\le \frac{\mu}{\mu-a }  \int^t_0 |v(s)| [C_0^\ast e^{-\mu (t-s)} + f(t-s)] ds,
\end{equation}
on $\delta \in D_{a/2}$.
Note that $ \int^t_0 |v(s)| e^{-\mu (t-s)} ds$ tends to zero by the dominated convergence theorem, as $\int^\infty_0 |v(s)| ds$ is finite (a direct consequence of expansion (\ref{vxexpan})). Also, the light tailed assumption in (\ref{assumption_light_tailed}) for $\tau_1$ entails that for all $j=1,\ldots,N$ one has $\int_0^t e^{-z_j s}f(t-s) ds= e^{-z_j t} \int_0^t e^{z_j s}f(s)ds\longrightarrow 0$ as $t\to\infty$. Similarly, $\int_0^t \eta(s) e^{-z_j s}f(t-s) ds \longrightarrow 0$ where $\eta(x)$ is defined by (\ref{function_eta}). Hence $\int_0^t |v(s)|f(t-s) ds $ tends to zero as $t\to\infty$. Then, from (\ref{ineg_I_2}) $I_2(t,\delta)$ uniformly converges to $0$ on $\delta \in D_{a/2}$. Thus, all in all, $\tilde{M}_{2}(t,\delta)$ converges uniformly towards $\chi_{2}(\delta)$ on $\delta \in D_{a/2}$.\hfill $\Box$

\noindent{\bf Proof of Theorem \ref{workload}.} Since $0\le -\left.\frac{\partial}{\partial \delta}\tilde{Z}(t,\delta)\right|_{\delta=0}=D(t)\le \sum_{i=1}^{N_t}L_i$ is integrable, it is possible to exchange differentiation with respect to $\delta$ and expectation and one has for all $t>0$
\begin{equation}\label{proof_interchange_expectation1}
-\left.\frac{\partial}{\partial \delta}\tilde{M}_{1}(t,\delta)\right|_{\delta=0}=-\left.\frac{\partial}{\partial \delta}\nbE[\tilde{Z}(t,\delta)]\right|_{\delta=0}= -\nbE\left[ \left.\frac{\partial}{\partial \delta}\tilde{Z}(t,\delta)\right|_{\delta=0}\right]= \nbE[D(t)].
\end{equation}
The main point in the proof is to study the limit in (\ref{proof_interchange_expectation1}) as $t\to\infty$. From Lemma \ref{Lemma_analytic}, utilizing the fact that $\delta\mapsto \tilde{M}_{1}(t,\delta)$ is analytic on the set $D_{a}$ where $a<\mu$ is arbitrary. Since by Lemma \ref{lemma_uniform_convergence}, $\tilde{M}_{1}(t,\delta)$ uniformly converges towards $\chi_{1}(\delta)$ on this set, a standard result in complex analysis states that the limiting function $\delta\mapsto\chi_{1}(\delta)$ is analytic on the same set. Hence it is in particular analytic at $\delta=0$ (which is known from its expression (\ref{CHIk11})) and, more importantly, one can interchange the order between differentiation and limit, i.e.
$$
\lim_{t\to\infty}\left.\frac{\partial}{\partial \delta}\tilde{M}_{1}(t,\delta)\right|_{\delta=0}=\left.\frac{\partial}{\partial \delta}\left[\lim_{t\to\infty}\tilde{M}_{1}(t,\delta)\right]\right|_{\delta=0}=\left.\frac{\partial}{\partial \delta}\chi_{1}(\delta)\right|_{\delta=0}.
$$
Expression of $ \chi_{1}(\delta)$ in the case $k=1$ is given in Corollary \ref{nmoment1}, Expression (\ref{CHIk11}) with $X_j=1$, yielding (\ref{limiting_expected_workload}).

Let us move on to the covariance of $D(t)$ and queue size $Z_1(t,0)$. One has $- \left.\frac{\partial}{\partial \delta}[Z_1(t,\delta)]^2\right|_{\delta=0}=2D(t)Z_1(t,0)$. Since the latter is integrable due to $D(t)Z_1(t,0)\le \left(\sum_{i=1}^{N_t}L_i\right) N_t$, as in (\ref{proof_interchange_expectation1}), interchanging expectation and differentiation results in
\[
-\left.\frac{\partial}{\partial \delta}\tilde{M}_{2}(t,\delta)\right|_{\delta=0}=2\nbE[D(t)Z_1(t,0)].
\]
The same argument of analyticity of $\delta\mapsto \tilde{M}_{2}(t,\delta) $ on $\delta\in D_{a/2}$ in Lemma \ref{Lemma_analytic}, coupled with the uniform convergence result as $t\to\infty$ in Lemma \ref{lemma_uniform_convergence} yields that $\lim_{t\to\infty}\left.\frac{\partial}{\partial \delta}\tilde{M}_{2}(t,\delta)\right|_{\delta=0}= \left.\frac{\partial}{\partial \delta} \chi_{2}(\delta)\right|_{\delta=0} $. Now the fact that $\lim_{t\to\infty} \tilde{M}_{1}(t,0)=\chi_{1}(0)$ and $ \lim_{t\to\infty}\left.\frac{\partial}{\partial \delta}\tilde{M}_{1}(t,\delta)\right|_{\delta=0}= \left.\frac{\partial}{\partial \delta} \chi_{1}(\delta)\right|_{\delta=0} $ implies
\begin{eqnarray}\label{end_expression_cov}
\lim_{t\to\infty}\Cov[D(t), Z_1(t,0)] &=&\lim_{t\to\infty}\nbE[D(t)Z_1(t,0)] - \nbE[D(t)]\nbE[Z_1(t,0)] \nonumber\\
&=&-\left.\frac{1}{2}\frac{\partial}{\partial \delta} \chi_{2}(\delta)\right|_{\delta=0}+\chi_{1}(0).\left. \frac{\partial}{\partial \delta} \chi_{1}(\delta)\right|_{\delta=0}.
\end{eqnarray}
Expression (\ref{CHIk1n}) with $X_j=1$ yields
$
\chi_{2}(\delta)=\frac{1}{\nbE[\tau_1]}\left( \frac{1}{\mu+2\delta}+\frac{\mu}{(\mu+\delta)^2}\frac{\L^\tau(\mu)}{1-\L^\tau(\mu)}\right),
$
and in turn,
\[
\left.\frac{\partial}{\partial \delta} \chi_{2}(\delta)\right|_{\delta=0}=-\left.\frac{1}{\nbE[\tau_1]}\left( \frac{2}{(\mu+2\delta)^2}+\frac{2\mu}{(\mu+\delta)^3}\frac{\L^\tau(\mu)}{1-\L^\tau(\mu)}\right)\right|_{\delta=0}=-\frac{2}{\mu^2\nbE[\tau_1]}\left(1+\frac{\L^\tau(\mu)}{1-\L^\tau(\mu)}\right).
\]
Hence, substitution of the above expression together with $\chi_{1}(\delta)$ obtained previously into (\ref{end_expression_cov}) yields (\ref{limiting_cov_workload}) for the limiting covariance.\hfill $\Box$


\section*{Acknowledgments}
The authors are very grateful to the anonymous referees for their careful reading and valuable comments on an earlier version of the manuscript which
have led to significant improvements in the paper. This work was supported by Joint Research Scheme France/Hong Kong Procore Hubert Curien grant No 35296 and F-HKU710/15T.

\bibliographystyle{alpha}



\end{document}